\documentclass[ams]{tac}
\usepackage[utf8]{inputenc}

\usepackage{amsmath,amssymb}
\usepackage{tikz}
\usetikzlibrary{calc}
\usepackage[hidelinks]{hyperref}
\usepackage[margin=1cm]{caption}
\usepackage[labelformat=simple]{subcaption}
\usepackage{cite}

\author{Sean Moss}
\eaddress{sean.moss@univ.ox.ac.uk}
\address{University College\\
  Oxford\\
  OX1 4BH\\
  UK}

\thanks{The author was supported by an EPSRC studentship at DPMMS, University of Cambridge, and by a Junior Research Fellowship at University College Oxford.}


\title{Another approach to the Kan-Quillen model structure}

\newtheorem{theorem}{Theorem}
\newtheorem{proposition}[theorem]{Proposition}

\newtheorem{lemma}[theorem]{Lemma}
\newtheorem{corollary}[theorem]{Corollary}
\theoremstyle{definition}
\newtheorem{definition}[theorem]{Definition}
\newtheorem{notation}[theorem]{Notation}
\newtheorem{remark}[theorem]{Remark}

\newcommand{\ex}{\mathop{\mathsf{Ex}}}
\newcommand{\exinf}{\mathop{\ex^\infty}}
\newcommand{\sd}{\mathop{\mathsf{sd}}}
\newcommand{\sset}{\mathsf{sSet}}
\newcommand{\nond}[1]{#1_{\textrm{\scriptsize n.d.}}}
\newcommand{\parent}{\mathop{\mathrm{{}P}}}
\newcommand{\ancpre}{\mathrel{\preceq_p}}
\newcommand{\anc}{\mathrel{\prec}}
\newcommand{\modanc}[1]{\mathrel{\prec_{#1}}}
\newcommand{\rank}{\mathop{\mathsf{rank}}}
\newcommand{\ord}{\mathsf{Ord}}
\newcommand{\children}[1]{#1_{\textrm{\scriptsize II}}}
\newcommand{\parents}[1]{#1_{\textrm{\scriptsize I}}}
\newcommand{\horn}{\Lambda}
\newcommand{\NN}{\mathbb{N}}
\newcommand{\id}{\mathrm{id}}

\newcommand{\ww}{\mathcal{W}}
\newcommand{\ff}{\mathcal{F}}
\newcommand{\cc}{\mathcal{C}}
\newcommand{\tf}{t\mathcal{F}}
\newcommand{\tc}{t\mathcal{C}}

\newcommand{\aaa}{\mathbb{A}}
\newcommand{\qarent}{\mathop{\mathrm{{}Q}}\nolimits}
\newcommand{\covered}[1]{\mathrel{\triangleleft_{#1}}}

\newcommand{\coorda}[1]{\coordinate (#1) at (-1.5454813220625092,0.15847321686606022);}
\newcommand{\coordb}[1]{\coordinate (#1) at (0.4141104721640331,-0.5914300969830179);}
\newcommand{\coordc}[1]{\coordinate (#1) at (1.1313708498984762,0.43295688011695754);}
\newcommand{\coordd}[1]{\coordinate (#1) at (0.0,2.0905007438022025);}

\begin{document}
% For TAC
\maketitle
\begin{abstract}
  By careful analysis of the embedding of a simplicial set into its image under Kan's $\exinf$ functor we obtain a new and combinatorial proof that it is a weak homotopy equivalence.
  Moreover, we obtain a presentation of it as a \emph{strong anodyne extension}.
  From this description we can quickly deduce some basic facts about $\exinf$ and hence provide a new construction of the Kan-Quillen model structure on simplicial sets, one which avoids the use of topological spaces or minimal fibrations.
\end{abstract}
% For JHRS
% \maketitle

\section{Introduction}\label{sec-introduction}

Model categories were introduced in \cite{QUILLEN1967}, in which a main example is $\sset$, the category of simplicial sets.
The given proof that it is an example is combinatorial but uses the theory of \emph{minimal fibrations}, thus it relies essentially on the axiom of choice.
There is also a topological approach (e.g.\ as in \cite{GOERSSJARDINE1999}), where the category of simplicial sets is related to a category of spaces via a nerve-realization adjunction.

The main contribution of this paper is a new proof of the existence of the Kan-Quillen model structure on $\sset$.
In Section \ref{sec-anodyne}, making the observation that many examples of trivial cofibrations are in fact \emph{strong anodyne extensions}, i.e.\ relative cell complexes of horn inclusions (not featuring retracts), we introduce the notion of \emph{regular pairing}, which is a convenient tool for presenting strong anodyne extensions.
The main example of a strong anodyne extension is the natural embedding of a simplicial set $X$ into its Kan fibrant replacement $\exinf X$ (introduced in \cite{KAN1957}), the properties of which are reviewed within this framework in Section \ref{sec-exinf}.
From this work, and a handful of elementary facts about simplicial homotopy, we quickly deduce the existence of the Kan-Quillen model structure in Section \ref{sec-model-structure}.
With the correct (classically equivalent) formulations of the classes of maps, we avoid using the axiom of choice though we still require the law of the excluded middle.

Finally, in Section \ref{sec-properness}, which stands independently from Sections \ref{sec-exinf} and \ref{sec-model-structure}, one further application of the concept of regular pairing is given, to prove a mild strengthening of the right-properness of this model structure.
Specifically, we show that the pullback of a horn inclusion along a fibration is a \emph{strong} anodyne extension.
Since we describe the strong anodyne extension explicitly, this means that given a pair of maps $Z \xrightarrow{g} Y \xrightarrow{f} X$ for which we are given the \emph{structure} of a Kan fibration for both $f$ and $g$ (i.e.\ a chosen solution for every lifting problem with horn inclusions) we may also equip the pushforward $\Pi_f(g) \in \sset/X$ with the \emph{structure} of a Kan fibration.

It is possible to view the notion of regular pairing used here as essentially a special case of Forman's discrete vector fields for CW-complexes \cite{forman-morse-theory-for-cell-complexes}.
Since an earlier version of this paper, our analysis of Kan's $\exinf$ functor has been adapted by others to stratified homotopy theory~\cite{douteau2018} and to a constructive setting~\cite{henry2019}.

\section{Anodyne extensions}\label{sec-anodyne}

The basic definitions for simplicial sets can be found in Chapter I of \cite{GOERSSJARDINE1999}.
To fix our notation, let us write $\Delta$ for the simplex category with its set of objects being $\{[n] \mid n \geq 0 \}$, where $[n] = \{0,1,\ldots,n\}$, and with maps being order-preserving functions.
We write $d^i$ or $d^i_n$ for the coface maps $[n-1] \to [n]$ and $s^i$ or $s^i_n$ for the codegeneracy maps $[n+1] \to [n]$.
We write $\sset$ for the presheaf category $[\Delta^\mathrm{op},\mathsf{Set}]$ of simplicial sets, $d_i$ and $s_i$ for the face and degeneracy maps of any given simplicial set, $\triangle^n$ for the standard $n$-simplex, and $\horn^n_k$ for its $k$\textsuperscript{th} horn.

Throughout this section $m : A \hookrightarrow B$ will denote a monomorphism in $\sset$.

\begin{definition}\label{def-sae}
  The monomorphism $m$ is a \emph{strong anodyne extension} if it admits an anodyne presentation.
  An \emph{anodyne presentation} for $m$ consists of an ordinal $\kappa$ and a $\kappa$-indexed increasing family of subcomplexes $(A^\alpha)_{\alpha \leq \kappa}$ of $B$, satisfying:
  \begin{itemize}
  \item $A^0 = A$, $A^\kappa = B$,
  \item for every non-zero limit ordinal $\lambda < \kappa$, $\bigcup_{\alpha < \lambda} A^\alpha = A^\lambda$,
  \item for every $\alpha < \kappa$, the inclusion $A^\alpha \hookrightarrow A^{\alpha+1}$ is a pushout of a coproduct of horn inclusions.
  \end{itemize}
\end{definition}

\begin{remark}\label{rem-sae}
  The definition of the class of strong anodyne extensions given above is the same as for the original class of anodyne extensions given in section IV.2 of \cite{GABRIELZISMAN1967} minus closure under retracts (also we have not restricted to \emph{countable} compositions, but \ref{cor-length-omega} below shows that this makes no difference).
  The usual factorization of a simplicial map as an anodyne extension followed by a Kan fibration using Quillen's small object argument (introduced in section II.\S 3 of \cite{QUILLEN1967}) gives us a factorization as a \emph{strong} anodyne extension followed by a Kan fibration.
  It follows easily that the usual notion of anodyne extension is recovered by closing the strong anodyne extensions under retracts.
\end{remark}

\begin{remark}\label{rem-ae-not-sae}
  Not every anodyne extension is strong.
  For example, consider the (contractible) simplicial set which is the quotient of two 2-simplices with edges and vertices identified as indicated:
  \begin{displaymath}
    \begin{tikzpicture}
      \coordinate (a) at (0,0);
      \coordinate (b) at (1,1.73);
      \coordinate (c) at (2,0);

      \draw[fill=black] (a) circle (0.4mm) node[below left] {\scriptsize $0$};
      \draw[fill=black] (b) circle (0.4mm) node[above] {\scriptsize $1$};
      \draw[fill=black] (c) circle (0.4mm) node[below right] {\scriptsize $1$};

      \draw[->,shorten >={0.4mm}]
      (a) edge node[auto,pos=0.4] {$a$} (b)
      (c) edge node[auto,pos=0.4,swap] {$b$} (b)
      (a) -- node[auto,swap] {$a$} (c);

      \begin{scope}[xshift=4cm]
        \coordinate (a) at (0,0);
        \coordinate (b) at (1,1.73);
        \coordinate (c) at (2,0);

        \draw[fill=black] (a) circle (0.4mm) node[below left] {\scriptsize $0$};
        \draw[fill=black] (b) circle (0.4mm) node[above] {\scriptsize $1$};
        \draw[fill=black] (c) circle (0.4mm) node[below right] {\scriptsize $1$};

        \draw[->,shorten >={0.4mm}]
        (a) edge node[auto,pos=0.4] {$a$} (b)
        (c) edge node[auto,pos=0.4,swap] {$b$} (b)
        (a) -- node[auto,swap] {$c$} (c);
      \end{scope}
    \end{tikzpicture}
  \end{displaymath}
  It is easy to check that the inclusion of the edge labelled $c$ is a monomorphism which does not admit an anodyne presentation.
  However, it must be an anodyne extension since its domain and codomain are both contractible (alternatively one can write the inclusion explicitly as a retract of a strong anodyne extension).
\end{remark}

Observe that, in the final clause of \ref{def-sae}, if $m : A^\alpha \to A^{\alpha+1}$ is somehow a pushout of a coproduct of horn inclusions then it is so in an essentially unique way.
Nondegenerate simplices in $A^{\alpha+1}$ but not in $A^\alpha$ come in pairs: one simplex which is not a face of any other nondegenerate simplex in $A^{\alpha+1}$ paired with the unique one of its codimension 1 faces also not in $A^\alpha$.
From such data one can infer which horns must feature in the description.
In fact, one can essentially describe the anodyne presentation using only such information.

\begin{notation}\label{not-nond}
  Let $X$ be any simplicial set, then by $\nond X$ we denote the set
  of nondegenerate simplices of $X$.
\end{notation}

\begin{definition}\label{def-pairing}
  A \emph{pairing} $\parent$ on $m$ is a partition of $\nond{B}\backslash\nond{A}$ into two disjoint sets $\parents{B}$ and $\children{B}$ together with a bijection $\parent : \children{B} \to \parents{B}$.
  Equivalently, a pairing consists of a partition of $\nond{B}\backslash\nond{A}$ into mutually disjoint subsets of size two where each of those subsets is equipped with a bijection to the set $\{\text{I},\text{II}\}$.
  We refer to elements of $\parents{B}$ as \emph{type I} simplices and elements of $\children{B}$ as \emph{type II} simplices.
\end{definition}

Above we briefly described how an anodyne presentation on $m : A \hookrightarrow B$ gives rise naturally to a pairing.
To be more precise: if $x \in \nond B \backslash \nond A$, then there exists a least $\delta$ such that $x \in \nond {A^\delta}$.
By `continuity' at limit ordinals, $\delta$ must be $\alpha + 1$ for some $\alpha$ with $x \notin \nond{A^\alpha}$.
As above, $x$ is either not a face of any other $y \in \nond{A^{\alpha+1}}\backslash \nond{A^\alpha}$, in which case we declare that it is a type I simplex, or if not it is a codimension 1 face of some $y$ which itself is not a face of any $z \in \nond{A^{\alpha+1}}\backslash \nond{A^\alpha}$, in which case we declare that $x$ is of type II.
It is clear that in the former case there is a unique type II simplex in $\nond{A^{\alpha+1}}\backslash\nond{A^\alpha}$ which is a proper (in fact codimension 1) face of $x$, and that in the latter case there is a unique type I simplex in $\nond{A^{\alpha+1}}\backslash\nond{A^\alpha}$ of which $x$ is a proper (in fact codimension 1) face.

We will now work towards characterizing those pairings that arise from anodyne presentations.
The pairing that arises from an anodyne presentation is clearly a proper pairing in the following sense.

\begin{definition}\label{def-proper-pairing}
  A \emph{proper pairing} on $m$ is a pairing $\parent$ on $m$ such that, for each $x \in B_{II}$, $x$ is a codimension 1 face of $\parent x$ in a unique way, i.e.\ $x = d_i(\parent x)$ for a unique $i$.
  In particular, for each $x \in B_{II}$, $\dim (\parent x) = \dim x + 1$.
\end{definition}

\begin{definition}\label{def-ancestral}
  Given a pairing $\parent$ on $m$, the \emph{ancestral relation} $\anc$ is the relation on $\children{B}$ given by $x_1 \anc x_2$ if and only if $x_1 \neq x_2$ and $x_1$ is a face of $\parent x_2$.
\end{definition}
If we are constructing an anodyne presentation giving rise to a particular pairing, then the predecessors of a simplex under the ancestral relation tell us which other simplices must already be present by the stage at which we choose to include that simplex.
Thus in \ref{def-ancestral}, $\children{B}$ really stands in for the quotient of $\nond{B} \backslash \nond{A}$ formed by identifying each type II simplex $x$ with $\parent x$.
\begin{definition}\label{def-ancestral-pre}
  Given a pairing $\parent$ on $m$, the \emph{ancestral preorder} $\ancpre$ is the smallest transitive relation on $\nond{B}\backslash\nond{A}$ satisfying
  \begin{enumerate}
  \item [(a)] whenever $w$ is a face of $z$ then $w \ancpre z$,
  \item [(b)] whenever $x \in \children{B}$ then $\parent x \ancpre x$.
  \end{enumerate}
\end{definition}
\begin{remark}
  The ancestral preorder is generated subject to reflexivity and transitivity by instances of the ancestral relation together with $\parent x \ancpre x$ and $x \ancpre \parent x$ for all $x \in \children{B}$.
\end{remark}

Recall that a binary relation $\triangleleft$ on a set $X$ is \emph{well-founded} if every subset $S \subseteq X$ is either empty or contains an $x \in S$ with $\forall y \in S. \neg(x \triangleleft y)$ (see, for example, \cite{johnstone-notes-on-logic-and-sets}).
This implies, and assuming the axiom of dependent choice is also implied by, the statement that there is no sequence $x_0,x_1,x_2,\ldots \in X$ such that $x_{n+1} \triangleleft x_n$ for all $n \in \mathbb N$.
Also recall that the strict ordering on any class of ordinals is well-founded.

\begin{proposition}\label{prop-anc-wf}
  The ancestral relation arising from an anodyne presentation on $m : A \hookrightarrow B$ is well-founded.
\end{proposition}
\begin{proof}
Define a map $\rank : \children{B} \to \ord$ valued in ordinals by sending $x$ to the supremum of all $\alpha$ such that $x \notin A^\alpha$.
Then $\rank$ is relation-preserving as a map $(\children{B},\anc) \to (\ord,<)$ because if $w,z \in \children{B}$ with $w$ a face of $\parent z$ and $\rank z = \alpha$, then $w,z,\parent z \in A^{\alpha+1}$, and since $z$ is the unique nondegenerate proper face of $\parent z$ in $A^{\alpha+1}$, we in fact have $w \in A^\alpha$ and hence $\rank w < \alpha = \rank z$.
Well-foundedness follows easily from the existence of a relation-preserving map into $(\ord,<)$.
\end{proof}

Let us emphasize that definitions \ref{def-ancestral} and \ref{def-ancestral-pre} rely only upon a pairing on $m$ and not upon an anodyne presentation.
Hence we can make the following definition.
\begin{definition}\label{def-structure}
  A \emph{regular pairing} on $m : A \hookrightarrow B$ is a proper pairing $\parent$ whose ancestral relation is well-founded.
\end{definition}

In the light of Proposition \ref{prop-anc-wf}, only a regular pairing can arise from an anodyne presentation.
In fact, regular pairings are precisely those that do arise from anodyne presentations.

\begin{proposition}\label{prop-structure-to-presentation}
  A regular pairing on $m : A \hookrightarrow B$ gives rise to an anodyne presentation for $m$ (of `length' at most $\omega$, i.e.\ for which $\kappa \leq \omega$).
\end{proposition}
\begin{proof}
Define $F : \children{B} \to \ord$ recursively by
\begin{displaymath}
  F(x) = \sup\{1,F(y) + 1 \mid y \anc x, y \in \children{B}\}.
\end{displaymath}
By $\anc$-induction, each $F(x)$ is a positive integer.
Now we can define $A^n$ by giving its nondegenerate simplices:
\begin{displaymath}
  \nond{(A^n)} = \nond{A} \cup \{x, \parent x \mid x \in \children{B}, F(x) \leq n\}.
\end{displaymath}
It is now easy to check the conditions of \ref{def-sae}.
\end{proof}

Compare the following corollary, which is immediate, with the consequence of the small object argument that every (not necessarily strong) anodyne extension is a retract of a monomorphism which admits a length $\omega$ anodyne presentation.

\begin{corollary}\label{cor-length-omega}
  Every strong anodyne extension admits an anodyne presentation of length $\omega$.
\end{corollary}

We see that anodyne presentations and regular pairings are essentially equivalent.
Indeed, because we will always verify that a pairing is regular by providing an explicit rank function, as in \ref{lem-modanc}, we could mechanically turn any proof using regular pairings into one using anodyne presentations directly.
However, in the examples considered here the pairing seems to be the more natural formulation for both the discovery and exposition of the explicit presentation.

Our formulation of regular pairing leads naturally to a length $\omega$ presentation, since the ancestral relation expresses the connection between a simplex and its lower-dimensional dependencies.
We might restrict the ancestral relation to just those type II simplices of some fixed dimension $n$, and write $\modanc n$ for this relation, and indeed we will actually check regularity of our pairings by Lemma \ref{lem-modanc} below.

\begin{lemma}\label{lem-modanc}
  Let $\parent$ be a proper pairing on $m : A \hookrightarrow B$.
  Then $\parent$ is regular if and only if there exists a function $\phi : \children{B} \to \mathbb N$ such that, for each $n$ and for all type II simplices $x$ and $y$ of dimension $n$, the implication $x \modanc n y \implies \phi(x) < \phi(y)$ holds.
\end{lemma}
\begin{proof}
  Straightforward since any $\anc$-descending chain must be eventually constant in dimension.
\end{proof}

However, the function $\phi$ in \ref{lem-modanc} is not in itself an expression of a length $\omega$ anodyne presentation: one must recursively use face operations to calculate the true rank of a simplex.
We can think of $\phi$ as specifying a length $\omega^2$ anodyne presentation: we use the first $\omega$ steps to attach all dimension 0 type II simplices, the next to attach all dimension 1 type II simplices, etc.
In the following trivial variant of \ref{lem-modanc}, the function $\psi$ may indeed be thought of as specifying a length $\omega$ presentation.
It is worth noting that each use of \ref{lem-modanc} in this paper (see the proofs of \ref{lem-horn-bdy}, \ref{prop-subd-sae}, \ref{thm-ex-sae}, and \ref{thm-proper}) could be replaced by a use of \ref{lem-modanc-var} using the translation $\psi(x) = \phi(x) + \dim x$, which would then give us the length $\omega$ presentation explicitly.
For clarity we give only the simpler proofs here.
\begin{lemma}\label{lem-modanc-var}
  Let $\parent$ be a proper pairing on $m : A \hookrightarrow B$.
  Then $\parent$ is regular if and only if there exists a function $\psi : \children{B} \to \mathbb N$ such that, for all type II simplices $x$ and $y$, the implication $x \anc y \implies \psi(x) < \psi(y)$ holds.
\end{lemma}

Let us demonstrate with a proof of a classical result.
Without the word `strong', the following result appears as proposition IV.2.2 in \cite{GABRIELZISMAN1967}.

\begin{proposition}\label{prop-anodyne-mono}
  Let $m : A \hookrightarrow B$ be a strong anodyne extension and $n : C \hookrightarrow D$ any monomorphism.
  Then
  \begin{displaymath}
    (B \times C) \cup (A \times D) \hookrightarrow B \times D
  \end{displaymath}
  is a strong anodyne extension.
\end{proposition}

This follows from Lemma \ref{lem-horn-bdy} below.
The deduction is essentially the standard argument: one shows that, for fixed $n : C \hookrightarrow D$, the class of monomorphisms $m : A \hookrightarrow B$ making the displayed inclusion strong anodyne is closed under coproducts, pushouts and transfinite composites.
Lemma \ref{lem-horn-bdy} shows that this class contains the horn inclusions (normally the word `strong' is omitted from its statement).
The proof uses the familiar triangulation of a product of simplices in terms of `shuffles' (see, for example, \cite[Definition 6.15]{MAY_simplicialobjects1967})

\begin{lemma}\label{lem-horn-bdy}
  Let $m \geq 1$, $n \geq 0$ and $0 \leq k \leq m$.
  Then
  \begin{displaymath}
    (\triangle^m \times \partial \triangle^n) \cup (\horn^m_k \times \triangle^n) \hookrightarrow \triangle^m \times \triangle^n
  \end{displaymath}
  is a strong anodyne extension.
\end{lemma}
\begin{proof}
  To begin with assume $0 \leq k < m$.
  The nondegenerate simplices of $B = \triangle^m \times \triangle^n$ may be identified with sequences $(\mu_i,\nu_i) \in \{0,\ldots,m\} \times \{0,\ldots,n\}$ of some length $N$, such that for $i < N$ we have $\mu_i \leq \mu_{i+1}$ and $\nu_i \leq \nu_{i+1}$ with at least one of these inequalities strict.
  The $i$\textsuperscript{th} face map is given by omitting the $i$\textsuperscript{th} point in the `walk' (counting from 0).
  If we think of $\{0,\ldots,m\}$ as labelling columns and $\{0,\ldots,n\}$ as labelling rows, then such a walk in the grid represents a simplex:
  \begin{itemize}
  \item in $\triangle^m \times \partial \triangle^n \subseteq \triangle^m \times \triangle^n$ if and only if some row is skipped,
  \item in $\horn^m_k \times \triangle^n \subseteq \triangle^m \times \triangle^n$ if and only if some column other than the $k$\textsuperscript{th} is skipped.
  \end{itemize}

  Consider a nondegenerate simplex $x \in \nond{B} \backslash \nond{A}$, where $A = (\triangle^m \times \partial \triangle^n) \cup (\horn^m_k \times \triangle^n)$.
  Every `move' in the walk is either $(+1,0)$, $(0,+1)$, $(+1,+1)$, or one of $(+2,0)$ and $(+2,+1)$ where the $k$\textsuperscript{th} column is skipped.
  Let us declare that $x$ is of type I if and only if at least one point of the walk is on the $k$\textsuperscript{th} column and moreover the last point of the walk which is on the $k$\textsuperscript{th} column is followed by a move of the form $(+1,0)$.
  In this situation, we pair $x$ with the simplex given by deleting the last point of the walk which is on the $k$\textsuperscript{th} column.
  In the case $0 < k < m$, we can describe this pairing with reference to Figure \ref{figure-middle-k}: we pair a simplex of the form \subref{subfig-r-r} with the simplex of the form \subref{subfig-rr} which agrees with it everywhere else, and similarly for the pairs \subref{subfig-r-u-r}-\subref{subfig-r-ur}, \subref{subfig-ur-r}-\subref{subfig-urr}, \subref{subfig-r-u-ell-u-r}-\subref{subfig-r-u-ell-ur}, and \subref{subfig-ur-ell-u-r}-\subref{subfig-ur-ell-ur}.
  In the case $k = 0$, we can describe the pairing with reference to Figure \ref{figure-k-0}: the pairs are of the form \subref{subfig-u-r}-\subref{subfig-ur} and \subref{subfig-}-\subref{subfig-r}

  This gives us a pairing, which is clearly proper in the sense of \ref{def-proper-pairing}.
  Let us check that it is regular.
  For a type II simplex $x \in \children{B}$, define $\phi(x)$ to be the number of points on the walk corresponding to $x$ which are on the $(k+1)$\textsuperscript{th} column.
  Suppose $x \modanc n y$.
  Then as $x$ is of type II it must be obtained from $\parent y$ by dropping either the last point on the $k$\textsuperscript{th} column or the first point on the $(k+1)$\textsuperscript{th} column.
  Since $x \anc y$ implies $x \neq y$ we must be in the latter situation, in which case $\phi(x) < \phi(y)$.

  To do the case $k = m$, note that we could just as well have proved the case $0 < k \leq m$ by looking at the \emph{first} point of a walk on the $k$\textsuperscript{th} column and checking whether it is \emph{preceded} by a move $(+1,0)$, etc.
  This gives two different pairings in the case $0 < k < m$.
\end{proof}

\begin{figure}[h]\centering%
  \begin{minipage}[t]{0.5\linewidth}
    \centering
    \begin{tikzpicture}
      \begin{scope}
        \draw[black!50!white,dotted] (-1.5,-1.5) grid (1.5,1.5);
        \draw[black!75!white,dashdotted,thick] (0,1.5) -- (0,-1.5);

        \draw[fill] (-1,0) circle (0.1);
        \draw[fill] (0,0) circle (0.1);
        \draw[fill] (1,0) circle (0.1);

        \draw (-1,0) -- (1,0);
        \draw[dash pattern=on 0.5mm off 1mm] (-1,0) -- +(-0.38,-0.38)
        (-1,0) -- +(0,-0.5)
        (-1,0) -- +(-0.5,0)

        (1,0) -- +(0.38,0.38)
        (1,0) -- +(0,0.5)
        (1,0) -- +(0.5,0)
        ;
      \end{scope}
    \end{tikzpicture}
    \subcaption{}\label{subfig-r-r}
  \end{minipage}%
  \begin{minipage}[t]{0.5\linewidth}
    \centering%
    \begin{tikzpicture}
      \begin{scope}
        \draw[black!50!white,dotted] (-1.5,-1.5) grid (1.5,1.5);
        \draw[black!75!white,dashdotted,thick] (0,1.5) -- (0,-1.5);

        \draw[fill] (-1,0) circle (0.1);
        \draw[fill] (1,0) circle (0.1);

        \draw (-1,0) -- (1,0);
        \draw[dash pattern=on 0.5mm off 1mm] (-1,0) -- +(-0.38,-0.38)
        (-1,0) -- +(0,-0.5)
        (-1,0) -- +(-0.5,0)

        (1,0) -- +(0.38,0.38)
        (1,0) -- +(0,0.5)
        (1,0) -- +(0.5,0)
        ;
      \end{scope}
    \end{tikzpicture}
    \subcaption{}\label{subfig-rr}
  \end{minipage}%
  \vskip 0.5cm
  \begin{minipage}[t]{0.25\linewidth}
    \centering%
    \begin{tikzpicture}
      \begin{scope}
        \draw[black!50!white,dotted] (-1.5,-0.8) grid (1.5,1.8);
        \draw[black!75!white,dashdotted,thick] (0,1.8) -- (0,-0.8);

        \draw[fill] (-1,0) circle (0.1);
        \draw[fill] (0,1) circle (0.1);
        \draw[fill] (0,0) circle (0.1);
        \draw[fill] (1,1) circle (0.1);

        \draw (-1,0) -- (0,0) -- (0,1) -- (1,1);
        \draw[dash pattern=on 0.5mm off 1mm] (-1,0) -- +(-0.38,-0.38)
        (-1,0) -- +(0,-0.5)
        (-1,0) -- +(-0.5,0)

        (1,1) -- +(0.38,0.38)
        (1,1) -- +(0,0.5)
        (1,1) -- +(0.5,0)
        ;
      \end{scope}
    \end{tikzpicture}
    \subcaption{}\label{subfig-r-u-r}
  \end{minipage}%
  \begin{minipage}[t]{0.25\linewidth}
    \centering%
    \begin{tikzpicture}
      \begin{scope}
        \draw[black!50!white,dotted] (-1.5,-0.8) grid (1.5,1.8);
        \draw[black!75!white,dashdotted,thick] (0,1.8) -- (0,-0.8);

        \draw[fill] (-1,0) circle (0.1);
        \draw[fill] (0,0) circle (0.1);
        \draw[fill] (1,1) circle (0.1);

        \draw (-1,0) -- (0,0) -- (1,1);
        \draw[dash pattern=on 0.5mm off 1mm] (-1,0) -- +(-0.38,-0.38)
        (-1,0) -- +(0,-0.5)
        (-1,0) -- +(-0.5,0)

        (1,1) -- +(0.38,0.38)
        (1,1) -- +(0,0.5)
        (1,1) -- +(0.5,0)        ;
      \end{scope}
    \end{tikzpicture}
    \subcaption{}\label{subfig-r-ur}
  \end{minipage}%
  \begin{minipage}[t]{0.25\linewidth}
    \centering%
    \begin{tikzpicture}
      \begin{scope}
        \draw[black!50!white,dotted] (-1.5,-0.8) grid (1.5,1.8);
        \draw[black!75!white,dashdotted,thick] (0,1.8) -- (0,-0.8);

        \draw[fill] (-1,0) circle (0.1);
        \draw[fill] (0,1) circle (0.1);
        \draw[fill] (1,1) circle (0.1);

        \draw (-1,0) -- (0,1) -- (1,1);
        \draw[dash pattern=on 0.5mm off 1mm] (-1,0) -- +(-0.38,-0.38)
        (-1,0) -- +(0,-0.5)
        (-1,0) -- +(-0.5,0)

        (1,1) -- +(0.38,0.38)
        (1,1) -- +(0,0.5)
        (1,1) -- +(0.5,0)        ;
      \end{scope}
    \end{tikzpicture}
    \subcaption{}\label{subfig-ur-r}
  \end{minipage}%
  \begin{minipage}[t]{0.25\linewidth}
    \centering%
    \begin{tikzpicture}
      \begin{scope}
        \draw[black!50!white,dotted] (-1.5,-0.8) grid (1.5,1.8);
        \draw[black!75!white,dashdotted,thick] (0,1.8) -- (0,-0.8);

        \draw[fill] (-1,0) circle (0.1);
        \draw[fill] (1,1) circle (0.1);

        \draw (-1,0) -- (1,1);
        \draw[dash pattern=on 0.5mm off 1mm] (-1,0) -- +(-0.38,-0.38)
        (-1,0) -- +(0,-0.5)
        (-1,0) -- +(-0.5,0)

        (1,1) -- +(0.38,0.38)
        (1,1) -- +(0,0.5)
        (1,1) -- +(0.5,0)        ;
      \end{scope}
    \end{tikzpicture}
    \subcaption{}\label{subfig-urr}
  \end{minipage}
  \vskip 0.5cm
  \begin{minipage}[t]{0.25\linewidth}
    \centering%
    \begin{tikzpicture}
      \begin{scope}
        \draw[black!50!white,dotted] (-1.5,-0.2) grid (1.5,1.8);
        \draw[black!75!white,dashdotted,thick] (0,1.8) -- (0,0);

        \draw[fill] (0,0) circle (0.1);
        \draw[fill] (0,1) circle (0.1);
        \draw[fill] (1,1) circle (0.1);

        \draw (0,-0.2) -- (0,0) -- (0,1) -- (1,1);
        \draw[dash pattern=on 0.5mm off 1mm]
        (1,1) edge +(0.38,0.38)
        edge +(0,0.5)
        edge +(0.5,0);

        \draw (0,-0.35) node[] {.};
        \draw (0,-0.5) node[] {.};
        \draw (0,-0.65) node[] {.};
      \end{scope}
      \begin{scope}[yshift=-2cm]
        \draw[black!50!white,dotted] (-1.5,-0.8) grid (1.5,1.2);
        \draw[black!75!white,dashdotted,thick] (0,1) -- (0,-0.8);

        \draw[fill] (-1,0) circle (0.1);
        \draw[fill] (0,0) circle (0.1);
        \draw[fill] (0,1) circle (0.1);

        \draw (-1,0) -- (0,0) -- (0,1) -- (0,1.2);
        \draw[dash pattern=on 0.5mm off 1mm]
        (-1,0) edge +(-0.38,-0.38)
        edge +(0,-0.5)
        edge +(-0.5,0);
      \end{scope}
    \end{tikzpicture}
    \subcaption{}\label{subfig-r-u-ell-u-r}
  \end{minipage}%
  \begin{minipage}[t]{0.25\linewidth}
    \centering%
    \begin{tikzpicture}
      \begin{scope}
        \draw[black!50!white,dotted] (-1.5,-0.2) grid (1.5,1.8);
        \draw[black!75!white,dashdotted,thick] (0,1.8) -- (0,0);

        \draw[fill] (0,0) circle (0.1);
        \draw[fill] (1,1) circle (0.1);

        \draw (0,-0.2) -- (0,0) -- (1,1);
        \draw[dash pattern=on 0.5mm off 1mm]
        (1,1) edge +(0.38,0.38)
        edge +(0,0.5)
        edge +(0.5,0);

        \draw (0,-0.35) node[] {.};
        \draw (0,-0.5) node[] {.};
        \draw (0,-0.65) node[] {.};
      \end{scope}
      \begin{scope}[yshift=-2cm]
        \draw[black!50!white,dotted] (-1.5,-0.8) grid (1.5,1.2);
        \draw[black!75!white,dashdotted,thick] (0,1) -- (0,-0.8);

        \draw[fill] (-1,0) circle (0.1);
        \draw[fill] (0,0) circle (0.1);
        \draw[fill] (0,1) circle (0.1);

        \draw (-1,0) -- (0,0) -- (0,1) -- (0,1.2);
        \draw[dash pattern=on 0.5mm off 1mm]
        (-1,0) edge +(-0.38,-0.38)
        edge +(0,-0.5)
        edge +(-0.5,0);
      \end{scope}
    \end{tikzpicture}
    \subcaption{}\label{subfig-r-u-ell-ur}
  \end{minipage}%
  \begin{minipage}[t]{0.25\linewidth}
    \centering%
    \begin{tikzpicture}
      \begin{scope}
        \draw[black!50!white,dotted] (-1.5,-0.2) grid (1.5,1.8);
        \draw[black!75!white,dashdotted,thick] (0,1.8) -- (0,0);

        \draw[fill] (0,0) circle (0.1);
        \draw[fill] (0,1) circle (0.1);
        \draw[fill] (1,1) circle (0.1);

        \draw (0,-0.2) -- (0,0) -- (0,1) -- (1,1);
        \draw[dash pattern=on 0.5mm off 1mm]
        (1,1) edge +(0.38,0.38)
        edge +(0,0.5)
        edge +(0.5,0);

        \draw (0,-0.35) node[] {.};
        \draw (0,-0.5) node[] {.};
        \draw (0,-0.65) node[] {.};
      \end{scope}
      \begin{scope}[yshift=-2cm]
        \draw[black!50!white,dotted] (-1.5,-0.8) grid (1.5,1.2);
        \draw[black!75!white,dashdotted,thick] (0,1) -- (0,-0.8);

        \draw[fill] (-1,0) circle (0.1);
        \draw[fill] (0,1) circle (0.1);

        \draw (-1,0) -- (0,1) -- (0,1.2);
        \draw[dash pattern=on 0.5mm off 1mm]
        (-1,0) edge +(-0.38,-0.38)
        edge +(0,-0.5)
        edge +(-0.5,0);
      \end{scope}
    \end{tikzpicture}
    \subcaption{}\label{subfig-ur-ell-u-r}
  \end{minipage}%
  \begin{minipage}[t]{0.25\linewidth}
    \centering%
    \begin{tikzpicture}
      \begin{scope}
        \draw[black!50!white,dotted] (-1.5,-0.2) grid (1.5,1.8);
        \draw[black!75!white,dashdotted,thick] (0,1.8) -- (0,0);

        \draw[fill] (0,0) circle (0.1);
        \draw[fill] (1,1) circle (0.1);

        \draw (0,-0.2) -- (0,0) -- (1,1);
        \draw[dash pattern=on 0.5mm off 1mm]
        (1,1) edge +(0.38,0.38)
        edge +(0,0.5)
        edge +(0.5,0);

        \draw (0,-0.35) node[] {.};
        \draw (0,-0.5) node[] {.};
        \draw (0,-0.65) node[] {.};
      \end{scope}
      \begin{scope}[yshift=-2cm]
        \draw[black!50!white,dotted] (-1.5,-0.8) grid (1.5,1.2);
        \draw[black!75!white,dashdotted,thick] (0,1) -- (0,-0.8);

        \draw[fill] (-1,0) circle (0.1);
        \draw[fill] (0,1) circle (0.1);

        \draw (-1,0) -- (0,1) -- (0,1.2);
        \draw[dash pattern=on 0.5mm off 1mm]
        (-1,0) edge +(-0.38,-0.38)
        edge +(0,-0.5)
        edge +(-0.5,0);
      \end{scope}
    \end{tikzpicture}
    \subcaption{}\label{subfig-ur-ell-ur}
  \end{minipage}
  \caption{The possible behaviours at the $k$\textsuperscript{th} column of a nondegenerate simplex $x \in \nond{B}\backslash\nond{A}$ from the proof of Lemma \ref{lem-horn-bdy}, in the case $0 < k < m$.}
  \label{figure-middle-k}
\end{figure}

\begin{figure}[h]\centering%
  \begin{minipage}[t]{0.25\linewidth}
    \centering%
    \begin{tikzpicture}
      \begin{scope}
        \draw[black!50!white,dotted] (0,-1.2) grid (1.5,0.8);
        \draw[black!75!white,dashdotted,thick] (0,0.8) -- (0,-1);

        \draw[fill] (1,0) circle (0.1);
        \draw[fill] (0,0) circle (0.1);
        \draw[fill] (0,-1) circle (0.1);

        \draw (0,-1.2) -- (0,-1) -- (0,0) --(1,0);
        \draw (0,-1.35) node[] {.};
        \draw (0,-1.5) node[] {.};
        \draw (0,-1.65) node[] {.};
        \draw[dash pattern=on 0.5mm off 1mm]
        (1,0) edge +(0.38,0.38)
        edge +(0,0.5)
        edge +(0.5,0)
        ;
      \end{scope}
      \begin{scope}[yshift=-2cm]
        \draw[black!50!white,dotted] (0,0) grid (1.5,0.2);

        \draw[fill] (0,0) circle (0.1);

        \draw (0,0) -- (0,0.2);
        ;
      \end{scope}
    \end{tikzpicture}
    \subcaption{}\label{subfig-u-r}
  \end{minipage}%
  \begin{minipage}[t]{0.25\linewidth}
    \centering%
    \begin{tikzpicture}
      \begin{scope}
        \draw[black!50!white,dotted] (0,-1.2) grid (1.5,0.8);
        \draw[black!75!white,dashdotted,thick] (0,0.8) -- (0,-1);

        \draw[fill] (1,0) circle (0.1);
        \draw[fill] (0,-1) circle (0.1);

        \draw (0,-1.2) -- (0,-1) -- (1,0);
        \draw (0,-1.35) node[] {.};
        \draw (0,-1.5) node[] {.};
        \draw (0,-1.65) node[] {.};
        \draw[dash pattern=on 0.5mm off 1mm]
        (1,0) edge +(0.38,0.38)
        edge +(0,0.5)
        edge +(0.5,0)
        ;
      \end{scope}
      \begin{scope}[yshift=-2cm]
        \draw[black!50!white,dotted] (0,0) grid (1.5,0.2);

        \draw[fill] (0,0) circle (0.1);

        \draw (0,0) -- (0,0.2);
        ;
      \end{scope}
    \end{tikzpicture}
    \subcaption{}\label{subfig-ur}
  \end{minipage}%
  \begin{minipage}[t]{0.25\linewidth}
    \centering%
    \begin{tikzpicture}
      \begin{scope}
        \draw[black!50!white,dotted] (0,0) grid (1.5,1.5);
        \draw[black!75!white,dashdotted,thick] (0,1.5) -- (0,0);

        \draw[fill] (1,0) circle (0.1);
        \draw[fill] (0,0) circle (0.1);

        \draw (0,0) -- (1,0);
        \draw[dash pattern=on 0.5mm off 1mm]
        (1,0) edge +(0.38,0.38)
        edge +(0,0.5)
        edge +(0.5,0)
        ;
      \end{scope}
    \end{tikzpicture}
    \subcaption{}\label{subfig-r}
  \end{minipage}%
  \begin{minipage}[t]{0.25\linewidth}
    \centering%
    \begin{tikzpicture}
      \begin{scope}
        \draw[black!50!white,dotted] (0,0) grid (1.5,1.5);
        \draw[black!75!white,dashdotted,thick] (0,1.5) -- (0,0);

        \draw[fill] (1,0) circle (0.1);

        \draw[dash pattern=on 0.5mm off 1mm]
        (1,0) edge +(0.38,0.38)
        edge +(0,0.5)
        edge +(0.5,0)
        ;
      \end{scope}
    \end{tikzpicture}
    \subcaption{}\label{subfig-}
  \end{minipage}
  \caption{The possible behaviours at the $0$\textsuperscript{th} column of a nondegenerate simplex $x \in \nond{B}\backslash\nond{A}$ from the proof of Lemma \ref{lem-horn-bdy}, in the case $k = 0$.}
  \label{figure-k-0}
\end{figure}

\begin{remark}
  In general, there may be many different pairings on the inclusion $(\triangle^m \times \partial \triangle^n) \cup (\horn^m_k \times \triangle^n) \hookrightarrow \triangle^m \times \triangle^n$.
  We can be more precise about the number of pairings if we count only a particular kind: by a \emph{based pairing} we mean a pairing on $(\triangle^m \times \partial \triangle^n) \cup (\horn^m_k \times \triangle^n) \hookrightarrow \triangle^m \times \triangle^n$ such that for any type II simplex $x$, if $d_i(\parent x) = x$ then the $i$\textsuperscript{th} vertex of $\parent x$ is $(k,a) \in \triangle^m \times \triangle^n$ for some $a$.
  In terms of walks, this means that operation $\parent$ adds an extra point on the $k$\textsuperscript{th} column.

  In the case $k = 0$ (and $k = m$), our proof of Lemma \ref{lem-horn-bdy} constructs the unique based pairing.
  To see this, observe that the minimal simplices of $\nond{B}\backslash\nond{A}$ are precisely those of the form \ref{subfig-ur} or \ref{subfig-}, and hence these must be of type II.
  In each case there is a unique simplex given by adding a point on the $0$\textsuperscript{th} column, and these have the forms \ref{subfig-u-r} and \ref{subfig-r} respectively.
  Pairing these exhausts all simplices, hence the based pairing is unique.

  In the case $0 < k < m$, the minimal simplices are those of the form \ref{subfig-rr}, \ref{subfig-urr}, or \ref{subfig-ur-ell-ur}.
  In the case \ref{subfig-rr}, there is no choice but to pair with a simplex of the form \ref{subfig-r-r}.
  In the case \ref{subfig-urr}, there is a choice between two simplices with forms \ref{subfig-r-ur} and \ref{subfig-ur-r}, but it is straightforward to see that whichever of the two not chosen must therefore be of type II and be paired with a simplex of the form \ref{subfig-r-u-r}.
  Similarly for \ref{subfig-ur-ell-ur}: we pair with either \ref{subfig-r-u-ell-ur} or \ref{subfig-ur-ell-u-r} and the remaining one must be paired with \ref{subfig-r-u-ell-u-r}.
  This exhausts all simplices, hence each based pairing amounts to a binary choice for every simplex of the form \ref{subfig-urr} or \ref{subfig-ur-ell-ur} in $\nond{B}\backslash\nond{A}$.
  Every such set of choices gives a regular pairing, which we can see by taking $\phi(x)$ to be the number of points on the walk corresponding to $x$ which are on either the $(k-1)$\textsuperscript{th} or $(k+1)$\textsuperscript{th} columns.
  Hence the number of based regular pairings can be calculated as $2^{L(m,n,k)}$, where $L(m,n,k)$ is the number of such simplices.
  It is straightforward to compute $L$ and to see that it is in fact independent of $k$.
\end{remark}

We present a second example in the use of regular pairings.
For the definition of the subdivision functor see \ref{def-exinf} below.

\begin{proposition}\label{prop-subd-sae}
  The subdivision functor $\sd : \sset \to \sset$ preserves strong anodyne extensions.
\end{proposition}
\begin{proof}
  Since $\sd$ is a left adjoint, it suffices to show that any horn inclusion is sent to a strong anodyne extension.
  Assume $n \geq 2$ as otherwise the proposition is trivial and for notational convenience assume $k = 0$, so we consider $\sd(\horn^n_0 \hookrightarrow \triangle^n)$.
  Thinking of the nondegenerate simplices of $\sd\triangle^n$ as strictly increasing sequences of non-empty subsets of $\{0,1,\ldots,n\}$, those not in $\sd \horn_0^n$ are those sequences which contain $\{1,\ldots,n\}$ or $\{0,1,\ldots,n\}$.

  Let us say that a simplex $(\sigma_1,\ldots,\sigma_r)$ in $\nond{(\sd \triangle^n)} \backslash \nond{(\sd \horn_0^n)}$ is of type I if and only if $\sigma_{i+1} = \sigma_i \cup \{0\}$ for some $0 \leq i < r$ where for notational purposes $\sigma_0 = \varnothing$.
  The corresponding type II simplex will be $(\sigma_1,\ldots,\widehat{\sigma_{i+1}},\ldots,\sigma_r)$, i.e.\ the same sequence but with its $(i+1)$\textsuperscript{th} element dropped.
  This defines a pairing which is clearly proper.

  Let us check that the pairing is regular.
  For a type II simplex $x = (\sigma_1,\ldots,\sigma_r)$, let $\phi(x)$ be the number of $i$ such that $0 \notin \sigma_i$.
  Now suppose that $x \modanc r y$.
  Then writing $y = (\tau_1,\ldots,\tau_a,\tau_{a+1},\ldots,\tau_r)$ where $0 \in \tau_{a+1} \backslash \tau_a$, we may write $\parent y$ as $(\tau_1,\ldots,\tau_a,\tau_a\cup\{0\},\tau_{a+1},\ldots,\tau_r)$.
  Then since $x$ is of type II it is given either by dropping $\tau_a$ or $\tau_{a}\cup\{0\}$ from $\parent y$, but since $x \neq y$ we are in the former situation, hence $\phi(x) < \phi(y)$.
\end{proof}

\section{Kan Fibrant Replacement}\label{sec-exinf}

Let us recall some definitions.
An account of $\exinf$ can be found in III.4 of \cite{GOERSSJARDINE1999}.

\begin{definition}\label{def-exinf}
  Recall that there is a functor $N : \mathsf{Cat} \to \sset$ sending a category to its \emph{nerve}.
  The standard simplex $\triangle^n$ arises as $N([n])$.
  Recall also that for the standard simplex $\triangle^n$ the set $\nond{(\triangle^n)}$ is naturally a poset and is isomorphic to the poset of non-empty subsets of $\{0,1,\ldots,n\}$.
  We define \emph{subdivision} $\sd : \Delta \to \sset$ by
  \begin{displaymath}
    \sd([n]) = N( \nond{(\triangle^n)}).
  \end{displaymath}

  We may extend subdivision to a functor $\sset \to \sset$ by left Kan extension.
  This functor has a right adjoint, called \emph{extension}, given by
  \begin{displaymath}
    (\ex X)_n = \sset(\sd \triangle^n,X).
  \end{displaymath}
  We think of this as the collection of all `binary pasting diagrams' in $X$ and we may refer to elements of $\ex X$ as \emph{(pasting) diagrams}.

  We also need the \emph{last-vertex map} $\bar\jmath_n : \sd \triangle^n \to \triangle^n$, the nerve of the map $\nond{(\triangle^n)} \to [n]$ given by
  \begin{displaymath}
    \{0,1,\ldots,n\} \supseteq \sigma \mapsto \max \sigma.
  \end{displaymath}
  The family $(\bar\jmath_n)$ is easily seen to be a natural transformation of functors $\Delta \to \sset$, and hence gives rise to a map $j_X : X \to \ex X$, where an $n$-simplex $\sigma : \triangle^n \to X$ is sent to $(\sigma \circ \bar\jmath_n : \sd \triangle^n \to X) \in (\ex X)_n$.
  This gives a natural transformation $\id \to \ex$.
  Finally, we define $\exinf X$ to be the colimit of the sequential diagram
  \begin{displaymath}
    \begin{tikzpicture}
      \node (a) at (0,0) {$X$};
      \node (b) at (3,0) {$\ex X$};
      \node (c) at (6,0) {$\ex (\ex X)$};
      \node (d) at (9,0) {\ldots};
      \draw[->,font=\footnotesize] (a) -- node[auto] {$j_X$} (b);
      \draw[->,font=\footnotesize] (b) -- node[auto] {$j_{\ex X}$} (c);
      \draw[->,font=\footnotesize] (c) -- node[auto] {$j_{\ex (\ex X)}$} (d);
    \end{tikzpicture}
  \end{displaymath}
  There is a map $X \to \exinf X$ natural in $X$ associated to the colimit, which we may also denote by $j_X$ when necessary.
\end{definition}

The main contribution in this section is a new proof that the map $j : X \to \ex X$ is a weak homotopy equivalence and indeed a strong anodyne extension.
Combined with the following classical result, this means that $\exinf$ is a functorial fibrant replacement for the Kan-Quillen model structure on simplicial sets.

\begin{proposition}\label{prop-exinf-kan}
  $\exinf X$ is a Kan complex.
\end{proposition}
\begin{proof}
  See, for example, III.4.8 in \cite{GOERSSJARDINE1999}.
\end{proof}

Our goal now is the following:
\begin{theorem}\label{thm-ex-sae}
  The map $j_X : X \to \ex X$ is a strong anodyne extension.
\end{theorem}

In fact, more is true than Proposition \ref{prop-exinf-kan}: any horn in $X$ has a filler in $\ex^2 X$.
Hence we can view \ref{prop-exinf-kan} and \ref{thm-ex-sae} as showing how to implement the horn-filling and subdivision-pasting operations in terms of one another.
Let us consider an example of how to make this more precise.
Say that $X$ admits \emph{subdivision-pasting} if $j_X : X \to \ex X$ admits a retraction $\ex X \to X$.
Now if $X$ admits subdivision-pasting, then (the strong form of) Proposition \ref{prop-exinf-kan} implies that $X$ is a Kan complex, i.e.\ $X$ admits the horn-filling operations.
Conversely, if $X$ is a Kan complex, then we can use \ref{thm-ex-sae} to get a retraction of $j_X : X \to \ex X$.

Observe that filling a horn of shape $\horn^{n+1}_k$ is like an $n$-ary operation on $n$-simplices that produces one $n$-simplex (and one $(n+1)$-simplex), whereas pasting together a diagram of shape $\sd \triangle^n$ is like an $(n+1)!$-ary operation on $n$-simplices that produces one $n$-simplex.
Thus we should expect that modelling subdivision-pasting via horn-filling will take an increasingly large number of steps as the dimension of the pasting diagram increases (whereas the converse only ever requires two steps, i.e.\ to fill any horn in $X$ we only need to pass to $\ex^2 X$).
Let us introduce a notion of `complexity' on the elements of $(\ex X)_n$ with respect to which we can reduce each subdivision pasting problem to a finite set of less complex ones after one horn-filling operation.

\begin{definition}\label{def-jk}
  Let $n, k \in \NN$ with $0 \leq k \leq n$.
  Define
  \begin{displaymath}
    j_n^k : \sd \triangle^n \to \sd \triangle^n
  \end{displaymath}
  to be the nerve of the unique map $\nond{(\triangle^n)} \to \nond{(\triangle^n)}$ that preserves binary joins and satisfies:
  \begin{displaymath}
    j_n^k(\{i\}) =
    \begin{cases}
      \{i\} & \text{ if $i \leq k$} \\
      \{0,1,\ldots,i\} & \text{ if $i > k$}.
    \end{cases}
  \end{displaymath}
\end{definition}

For example, $j_n^n$ is the identity and $j_n^0$ is the last vertex map $\bar\jmath_n : \sd \triangle^n \to \triangle^n$ composed with the appropriate inclusion $\triangle^n \to \sd \triangle^n$.

\begin{figure}[h]
  \centering
  \begin{tikzpicture}
    \begin{scope}[xshift=0cm]
      \coordinate (20) at (-1.25,0);
      \coordinate (21) at (1.25,0);
      \coordinate (22) at (0,2.165);
      \coordinate (201) at (barycentric cs:20=1,21=1);
      \coordinate (202) at (barycentric cs:20=1,22=1);
      \coordinate (212) at (barycentric cs:21=1,22=1);
      \coordinate (2012) at (barycentric cs:20=1,21=1,22=1);
      \draw
      (20) node[below left] {\scriptsize $0$}
      (21) node[below right] {\scriptsize $1$}
      (22) node[above] {\scriptsize $2$}
      ;
      \begin{scope}[black!25!white]
        \draw[fill]
        (21) circle (0.1em)
        (22) circle (0.1em)
        (202) circle (0.1em)
        (212) circle (0.1em)
        ;
        \path[->,shorten >=0.15em]
        (20) edge (202)
        (21) edge (201)
        (21) edge (212)
        (21) edge (2012)
        (22) edge (202)
        (22) edge (212)
        (22) edge (2012)
        (202) edge (2012)
        (212) edge (2012)
        ;
      \end{scope}
      \begin{scope}[black!60!white]
        \coordinate (21) at (0.3,0);
        \coordinate (22) at (0,1.15);
        \coordinate (202) at (barycentric cs:20=1,22=5);
        \coordinate (212) at (barycentric cs:21=1,22=2.5,212=0.7);
        \draw[fill]
        (21) circle (0.1em)
        (22) circle (0.1em)
        (202) circle (0.1em)
        (212) circle (0.1em)
        ;
        \path[->,shorten >=0.15em]
        (20) edge (202)
        (21) edge (201)
        (21) edge (212)
        (21) edge (2012)
        (22) edge (202)
        (22) edge (212)
        (22) edge (2012)
        (202) edge (2012)
        (212) edge (2012)
        ;
      \end{scope}
      \draw[fill]
      (20) circle (0.1em)
      (201) circle (0.1em)
      (2012) circle (0.1em)
      ;
      \path[->,shorten >=0.15em]
      (20) edge (201)
      (201) edge (2012)
      (20) edge (2012)
      ;
    \end{scope}
    \begin{scope}[xshift=4cm]
      \coordinate (20) at (-1.25,0);
      \coordinate (21) at (1.25,0);
      \coordinate (22) at (0,2.165);
      \coordinate (201) at (barycentric cs:20=1,21=1);
      \coordinate (202) at (barycentric cs:20=1,22=1);
      \coordinate (212) at (barycentric cs:21=1,22=1);
      \coordinate (2012) at (barycentric cs:20=1,21=1,22=1);
      \draw
      (20) node[below left] {\scriptsize $0$}
      (21) node[below right] {\scriptsize $1$}
      (22) node[above] {\scriptsize $2$}
      ;
      \begin{scope}[black!25!white]
        \draw[fill]
        (22) circle (0.1em)
        (202) circle (0.1em)
        (212) circle (0.1em)
        ;
        \path[->,shorten >=0.15em]
        (20) edge (202)
        (21) edge (212)
        (22) edge (202)
        (22) edge (212)
        (22) edge (2012)
        (202) edge (2012)
        (212) edge (2012)
        ;
      \end{scope}
      \begin{scope}[black!60!white]
        \coordinate (22) at (0,1.15);
        \coordinate (202) at (barycentric cs:20=1,22=5);
        \coordinate (212) at (barycentric cs:21=1,22=5);
        \draw[fill]
        (22) circle (0.1em)
        (202) circle (0.1em)
        (212) circle (0.1em)
        ;
        \path[->,shorten >=0.15em]
        (20) edge (202)
        (21) edge (212)
        (21) edge (2012)
        (22) edge (202)
        (22) edge (212)
        (22) edge (2012)
        (202) edge (2012)
        (212) edge (2012)
        ;
      \end{scope}
      \draw[fill]
      (20) circle (0.1em)
      (21) circle (0.1em)
      (201) circle (0.1em)
      (2012) circle (0.1em)
      ;
      \path[->,shorten >=0.15em]
      (20) edge (201)
      (21) edge (201)
      (201) edge (2012)
      (20) edge (2012)
      (21) edge (2012)
      ;
    \end{scope}
    \begin{scope}[xshift=8cm]
      \coordinate (20) at (-1.25,0);
      \coordinate (21) at (1.25,0);
      \coordinate (22) at (0,2.165);
      \coordinate (201) at (barycentric cs:20=1,21=1);
      \coordinate (202) at (barycentric cs:20=1,22=1);
      \coordinate (212) at (barycentric cs:21=1,22=1);
      \coordinate (2012) at (barycentric cs:20=1,21=1,22=1);
      \draw[fill]
      (20) circle (0.1em) node[below left] {\scriptsize $0$}
      (21) circle (0.1em) node[below right] {\scriptsize $1$}
      (22) circle (0.1em) node[above] {\scriptsize $2$}
      (201) circle (0.1em)
      (202) circle (0.1em)
      (212) circle (0.1em)
      (2012) circle (0.1em)
      ;
      \path[->,shorten >=0.15em]
      (20) edge (201)
      (20) edge (202)
      (20) edge (2012)
      (21) edge (201)
      (21) edge (212)
      (21) edge (2012)
      (22) edge (202)
      (22) edge (212)
      (22) edge (2012)
      (201) edge (2012)
      (202) edge (2012)
      (212) edge (2012)
      ;
    \end{scope}
  \end{tikzpicture}
  \caption{$j_2^k$ for $k \in \{0,1,2\}$.}
  \label{fig:defn-j-n-k}
\end{figure}

It is easily seen that $j_n^k$ is idempotent, (i.e.\ $j_n^k \circ j_n^k = j_n^k$), so that $(- \circ j_n^k) : (\ex X)_n \to (\ex X)_n$ is idempotent for any $X$.
Taking the image --- or, equivalently, the set of fixed points --- of each function $(-\circ j_n^k)$ gives us a `filtration' of $(\ex X)_n$.

\begin{notation}\label{notn-filtration-ex-x}
  Fix some simplicial set $X$.
  Let
  \begin{displaymath}
    J_n^k = \{ \sigma \in (\ex X)_n \mid \sigma = \sigma \circ j_n^k\}.
  \end{displaymath}
\end{notation}

It is easily checked that $j_n^k \circ j_n^l = j_n^k$ for $k \leq l$, so that we have a chain of subsets
\begin{displaymath}
  j_X(X_n) = J_n^0 \subseteq J_n^1 \subseteq \ldots \subseteq J_n^{n-1} \subseteq J_n^n = (\ex X)_n.
\end{displaymath}

Roughly speaking, each $J_n^k$ is the collection of those pasting diagrams of shape $\sd \triangle^n$ in $X$ which consist of at most $(k+1)!$ nondegenerate $n$-simplices.
To be slightly more precise, each pasting diagram in $J_n^k$ has arbitrary $n$-simplices in the $(k+1)!$ operand positions corresponding to those permutations of $(0,1,\ldots,n)$ which fix each of the tail elements $k+1,\ldots,n$, and only degenerate simplices in the remaining operand positions.
(To be completely precise, we would have to say in which directions these latter simplices are degenerate.)

Now we introduce a generalization of the codegeneracy maps whose role is to show us how to reduce the construction of one element of $J_n^{k+1}$ to the construction of $(k+2)$ elements of $J_n^k$ followed by a horn-filling.

\begin{definition}\label{def-rk}
  Let $n,k \in \NN$ with $0 \leq k \leq n$ and define
  \begin{displaymath}
    r_n^k : \sd \triangle^{n+1} \to \sd \triangle^n
  \end{displaymath}
  to be the nerve of the unique map $\nond{(\triangle^{n+1})} \to \nond{(\triangle^n)}$ that preserves binary joins and satisfies:
  \begin{displaymath}
    r_n^k(\{i\}) =
    \begin{cases}
      \{i\} & \text{ if $i \leq k$} \\
      \{0,1,\ldots,k\} & \text{ if $i = k+1$} \\
      \{i-1\} & \text{ if $i > k+1$}.
    \end{cases}
  \end{displaymath}
  These maps are almost the subdivided codegeneracy maps $\sd (s^k)$, except for their behaviour at $k+1$.
  Indeed, $r_n^0 = \sd s_n^0$.
\end{definition}

\begin{figure}[h]
  \centering
  \begin{tikzpicture}
    \begin{scope}[xshift=0cm]
      \coorda{t0}
      \coordb{t1}
      \coordc{t2}
      \coordd{t3}
      \coordinate (t01) at (barycentric cs:t0=1,t1=1);
      \draw[fill=black]
      (t0) circle (0.1em) node[below left] {\scriptsize $0$}
      (t1) circle (0.1em) node[below] {\scriptsize $1$}
      (t2) circle (0.1em) node[below right] {\scriptsize $3$}
      (t3) circle (0.1em) node[above] {\scriptsize $2$}
      ;
      \draw (t0) -- (t1) -- (t2) -- (t3) -- (t0) -- (t2) (t1) -- (t3);
      \draw[dashed,->] (t3) -- (t01);
      \draw[dotted] (t01) edge (t2);
      \coordinate (p0) at ($ (t0) + (0,-2) $);
      \coordinate (p1) at ($ (t1) + (0,-2) $);
      \coordinate (p2) at ($ (t2) + (0,-2) $);
      \draw[fill=black]
      (p0) circle (0.1em) node[below left] {\scriptsize $0$}
      (p1) circle (0.1em) node[below] {\scriptsize $1$}
      (p2) circle (0.1em) node[below right] {\scriptsize $2$}
      ;
      \coordinate (p01) at (barycentric cs:p0=1,p1=1);
      \draw (p0) -- (p1) -- (p2) edge (p0) edge (p01);
    \end{scope}
    \begin{scope}[xshift=5cm]
      \coorda{t0}
      \coordb{t1}
      \coordc{t2}
      \coordd{t3}
      \coordinate (t012) at (barycentric cs:t0=1,t1=1,t2=1);
      \draw[fill=black]
      (t0) circle (0.1em) node[below left] {\scriptsize $0$}
      (t1) circle (0.1em) node[below] {\scriptsize $1$}
      (t2) circle (0.1em) node[below right] {\scriptsize $2$}
      (t3) circle (0.1em) node[above] {\scriptsize $3$}
      ;
      \draw (t0) -- (t1) -- (t2) -- (t3) -- (t0) -- (t2) (t1) -- (t3);
      \draw[dashed,->] (t3) -- (t012);
      \path[dotted] (t012) edge (t0) edge (t1) edge (t2);
      \coordinate (p0) at ($ (t0) + (0,-2) $);
      \coordinate (p1) at ($ (t1) + (0,-2) $);
      \coordinate (p2) at ($ (t2) + (0,-2) $);
      \draw[fill=black]
      (p0) circle (0.1em) node[below left] {\scriptsize $0$}
      (p1) circle (0.1em) node[below] {\scriptsize $1$}
      (p2) circle (0.1em) node[below right] {\scriptsize $2$}
      ;
      \coordinate (p012) at (barycentric cs:p0=1,p1=1,p2=1);
      \draw (p0) -- (p1) -- (p2) -- (p0) -- (p012) edge (p1) edge (p2);
    \end{scope}
  \end{tikzpicture}
  \caption{$r_2^1$ and $r_2^2$.}
  \label{fig:defn-r-2-k}
\end{figure}

\begin{figure}[h]
  \centering
  \begin{tikzpicture}
    \begin{scope}[xshift=0cm]
      \coorda{t0}
      \coordb{t1}
      \coordc{t2}
      \coordd{t3}
      \coordinate (t01) at (barycentric cs:t0=1,t1=1);
      \fill[gray,fill opacity=0.15] (t2) -- (t01) -- (t3) -- cycle;
      \draw
      (t0) -- (t1) -- (t2) -- (t3) -- (t0) -- (t2)
      (t1) -- (t3)
      (t2) -- (t01) -- (t3)
      ;
      \draw[fill=black]
      (t0) circle (0.1em)
      (t1) circle (0.1em)
      (t2) circle (0.1em)
      (t3) circle (0.1em)
      (t01) circle (0.1em)
      ;
    \end{scope}
    \begin{scope}[xshift=5cm]
      \coorda{t0}
      \coordb{t1}
      \coordc{t2}
      \coordd{t3}
      \coordinate (t012) at (barycentric cs:t0=1,t1=1,t2=1);
      \fill[gray,fill opacity=0.15]
      (t3) -- (t012) -- (t2) -- cycle
      (t3) -- (t012) -- (t1) -- cycle
      (t3) -- (t012) -- (t0) -- cycle
      ;
      \draw
      (t0) -- (t1) -- (t2) -- (t3) -- (t0) -- (t2)
      (t1) -- (t3)
      (t3) -- (t012) edge (t1) edge (t0) edge (t2)
      ;
      \draw[fill=black]
      (t0) circle (0.1em)
      (t1) circle (0.1em)
      (t2) circle (0.1em)
      (t3) circle (0.1em)
      (t012) circle (0.1em)
      ;
    \end{scope}
    \begin{scope}[xshift=10cm]
      \coorda{t0}
      \coordb{t1}
      \coordc{t2}
      \coordd{t3}
      \coordinate (t0123) at (barycentric cs:t0=1,t1=1,t2=1,t3=1);
      \fill[gray,fill opacity=0.15]
      (t3) -- (t0123) -- (t2) -- cycle
      (t3) -- (t0123) -- (t1) -- cycle
      (t3) -- (t0123) -- (t0) -- cycle
      (t2) -- (t0123) -- (t0) -- cycle
      (t2) -- (t0123) -- (t1) -- cycle
      (t1) -- (t0123) -- (t0) -- cycle
      ;
      \draw (t0) -- (t1) -- (t2) -- (t3) -- (t0) -- (t2)
      (t1) -- (t3)
      (t3) -- (t0123) edge (t1) edge (t0) edge (t2)
      ;
      \draw[fill=black]
      (t0) circle (0.1em)
      (t1) circle (0.1em)
      (t2) circle (0.1em)
      (t3) circle (0.1em)
      (t0123) circle (0.1em)
      ;
    \end{scope}
  \end{tikzpicture}
  \caption{The images of the 0-, 1- and  2-dimensional faces of $\triangle^4$ in $\triangle^3$ under the maps $r_3^1$, $r_3^2$ and $r_3^3$. Shading is used to highlight the images of 2-dimensional faces not coincident with faces of $\triangle^3$.}
  \label{fig:defn-r-3-k}
\end{figure}

We can give an explicit description of a pairing on $j_X : X \to \ex X$.
Let us say that a nondegenerate $(n+1)$-simplex $\rho \notin J_{n+1}^0$ is of type I if and only if it is of the form $\tau \circ r_n^h$ with $\tau \in J_n^h \backslash J_n^{h-1}$, for some $1 \leq h \leq n$.
Hence, for a type II $n$-simplex $\sigma$, we let $\parent \sigma = \sigma \circ r_n^k$ where $k$ is determined by $\sigma \in J_n^k \backslash J_n^{k-1}$.
It remains to check that this is indeed a well-defined proper pairing and that it is regular.
Even though there are a few details to check, the specification of the pairing is very simple and, as we shall see in \ref{prop-reflects}, useful.
Let us record here the main equations we shall use in the course of the proof. Besides the proof of these, we shall not need to refer again to the details of Definitions \ref{def-jk} and \ref{def-rk}.

\begin{lemma}\label{lem-jr-eqns}
  The following equations hold when the given inequalities are satisfied.
    \begin{align}
    r_n^k \circ (\sd d^{k+1}) & = \mathrm{id}_{\sd \triangle^n}  & & 0 \leq k \leq n \label{eq:r-has-face} \\
    j_n^k \circ r_n^k \circ (\sd d^i) \circ j_n^{k-1} & = j_n^k \circ r_n^k \circ (\sd d^i) & & \mbox{$0 \leq i \leq k \leq n$ and $1 \leq k$} \label{eq:r-lower-faces} \\
    r_n^k \circ (\sd d^i) & = (\sd d^{i-1}) \circ r_{n-1}^k & & \mbox{$0 \leq k$ and $k+2 \leq i \leq n+1$} \label{eq:r-upper-faces-r} \\
    r_n^k \circ j_{n+1}^h & = j^h_n \circ r_n^k & & \mbox{$0 \leq h \leq k \leq n$} \label{eq:r-j-j-r} \\
    j_n^k \circ (\sd d^i) \circ j_{n-1}^k & = j_n^k \circ (\sd d^i) & & \mbox{$0 \leq i \leq n$ and $0 \leq k \leq n-1$} \label{eq:j-face-j} \\
    j_{n-1}^h \circ r_{n-1}^k & = j_{n-1}^h \circ (\sd s^k) & & 0 \leq h < k \leq n-1 \label{eq:j-r-deg} \\
    j_{n-1}^k \circ r_{n-1}^k \circ r_n^k & = j_{n-1}^k \circ r_{n-1}^k \circ (\sd s^{k+1}) & & 0 \leq k \leq n-1 \label{eq:j-r-r-deg-2} \\
    (\sd s^h) \circ j_{n+1}^{k+1} \circ r_{n+1}^{k+1} & = j_n^k \circ r_n^k \circ (\sd s^h) & & 0 \leq h \leq k \leq n \label{eq:s-j-r-j-r-s-1} \\
    (\sd s^h) \circ j_{n+1}^k \circ r_{n+1}^k & = j_n^k \circ r_n^k \circ (\sd s^{h+1}) & & 0 \leq k \leq h \leq n \label{eq:s-j-r-j-r-s-2}
  \end{align}
\end{lemma}
\begin{proof}
  These are easy calculations at the level of posets prior to taking the nerve and using the fact that all the maps involved preserve joins of non-empty sets.
  They have been verified in the Coq proof assistant~\cite{my-github-repo}.
\end{proof}

We begin by showing that the operation $\parent$ is injective.
\begin{lemma}\label{lem-i-in-a-unique-way}
  There is at most one way to write any $\sigma \in (\ex X)_n$ as $\sigma = \tau \circ r_n^k$ where $k \geq 1$ and $\tau \in J_n^k\backslash J_n^{k-1}$.
\end{lemma}
\begin{proof}
  Suppose $\sigma = \tau_i \circ r_n^{k_i}$ with $\tau_i \in J_n^{k_i}\backslash J_n^{k_i-1}$ for $i=1,2$.
  Then by \eqref{eq:r-j-j-r}, we have $\sigma \in J_{n+1}^{k_i}$ for $i=1,2$.
  By \eqref{eq:r-has-face} we have $\tau_i = \sigma \circ (\sd d^{k_i+1})$, so it is enough to show that $k_1 = k_2$.
  Now
  \begin{align*}
    \tau_2 \circ j_n^{k_1} & = \sigma \circ (\sd d^{k_2+1}) \circ j_n^{k_1} \\
                           & = \sigma \circ j_{n+1}^{k_1} \circ (\sd d^{k_2+1}) \circ j_n^{k_1} \\
                           & = \sigma \circ j_{n+1}^{k_1} \circ (\sd d^{k_2+1}) \\
                           & = \sigma \circ (\sd d^{k_2+1}) \\
                           & = \tau_2
  \end{align*}
  using \eqref{eq:j-face-j}.
  Hence $\tau_2 \in J_n^{k_1}$, hence $k_2 \leq k_1$.
  But by a similar argument $k_1 \leq k_2$.
\end{proof}

The next lemma shows that the image of the operation $\parent$ does indeed contain all type I simplices.
\begin{lemma}\label{lem-all-have-a-type}
  If $\rho \in (\ex X)_{n+1}\backslash J_{n+1}^0$ is nondegenerate and of the form $\sigma \circ r_n^k$ where $k \geq 1$ and $\sigma \in J_n^k \backslash J_n^{k-1}$ (i.e.\ $\rho$ is a type I $(n+1)$-simplex), then $\sigma$ is of type II.
\end{lemma}
\begin{proof}
  If $\sigma$ is degenerate, then it follows immediately from \eqref{eq:s-j-r-j-r-s-1} and \eqref{eq:s-j-r-j-r-s-2} that $\rho$ is as well.
  Suppose for contradiction that $\sigma$ is nondegenerate yet not of type II.
  Then either $\sigma \in J_n^0$ and hence, by \eqref{eq:r-j-j-r}, $\rho \in J_{n+1}^0$, or there is an expression $\sigma = \tau \circ r_{n-1}^h$ with $h \geq 1$ and $\tau \in J_{n-1}^h \backslash J_{n-1}^{h-1}$.
  Since $\tau = \sigma \circ (\sd d^{h+1})$ by \eqref{eq:r-has-face}, it follows by \eqref{eq:j-face-j} that $\tau \circ j_{n-1}^k = \tau$, hence $h \leq k$.
  Now applying \eqref{eq:r-j-j-r} we see that $\sigma \circ j_n^h = \sigma$ hence $k \leq h$, so indeed $k = h$.
  It now follows from \eqref{eq:j-r-r-deg-2} that $\rho = \tau \circ r_{n-1}^k \circ r_n^k$ is degenerate.
\end{proof}

Now we check that $\parent$ only takes values in type I simplices and is a proper pairing.
\begin{lemma}\label{lem-only-type-i-proper-pairing}
  Let $\sigma \in (\ex X)_n$ be type II.
  Then there is a unique $0 \leq i \leq n+1$ such that $d^i_{\ex X}(\parent\sigma) = \sigma$.
  In particular, $\parent\sigma$ is nondegenerate.
  In addition, $\parent\sigma \notin J_{n+1}^0$, hence $\parent\sigma$ is a type I simplex.
\end{lemma}
\begin{proof}
  Suppose that $\sigma \in J_n^k \backslash J_n^{k-1}$ and consider three cases for $i$.
  Firstly, if $i = k+1$, we have $\sigma = (\parent \sigma) \circ (\sd d^{k+1}) = d_{\ex X}^i(\parent \sigma)$ by \eqref{eq:r-has-face}.
  Secondly, if $i \leq k$, then by \eqref{eq:r-lower-faces}, $(\parent \sigma) \circ (\sd d^i) \in J_n^{k-1}$, so it cannot be equal to $\sigma$.
  Thirdly, if $i \geq k+2$, then by \eqref{eq:r-upper-faces-r}, $(\parent \sigma) \circ (\sd d^i) = \tau \circ r_{n-1}^k$ where $\tau = \sigma \circ (\sd d^{i-1})$.
  By \eqref{eq:j-face-j}, $\tau \in J_{n-1}^k$.
  If $\tau \notin J_{n-1}^{k-1}$, then $(\parent \sigma) \circ (\sd d^i)$ cannot be a type II simplex, so cannot be equal to $\sigma$.
  If $\tau \in J_{n-1}^{k-1}$, then by \eqref{eq:j-r-deg}, $(\parent \sigma) \circ (\sd d^i)$ is degenerate, so cannot be equal to $\sigma$.
  Finally, observe that for any type II $n$-simplex $\sigma$, the equation $(\parent \sigma) \circ j_{n+1}^0 = \parent \sigma$ would imply (by \eqref{eq:r-has-face} and \eqref{eq:j-face-j}) that $\sigma \circ j_n^0 = \sigma$, so indeed $\parent \sigma \notin J_{n+1}^0$.

\end{proof}
The proof above illustrates how the pairing $\parent$ works because we chose the map $r^k_n$ such that, for $\sigma \in J_n^k \backslash J_n^{k-1}$, if $i < k+1$ the faces $d^i_{\ex X}(\parent\sigma)$ are of lower `complexity' than $\sigma$ and if $i > k+1$ the faces are `degenerate', albeit possibly only in the more general sense of being $\tau \circ r_n^k = \parent \tau$ for some $\tau$.
This same feature is what allows us to prove regularity.

\begin{proof}[of Theorem \ref{thm-ex-sae}.]
  The foregoing lemmas show that we do indeed have a proper pairing $\parent$.
  All that remains is to check that the pairing is regular.
  Given a type II $n$-simplex $\sigma$, let $\phi(\sigma) = k$ where $\sigma \in J_n^k \backslash J_n^{k-1}$.
  Now suppose that $\xi \modanc n \sigma$.
  We will show that $\phi(\xi) < \phi(\sigma)$.
  Still writing $k = \phi(\sigma)$, we must have $\xi = d_{\ex X}^i(\parent \sigma)$ for some $i \neq k + 1$ since $\xi \neq \sigma$.
  Firstly, let us suppose that $i \geq k+2$.
  Then, by \eqref{eq:r-upper-faces-r} and as in the proof of Lemma \ref{lem-only-type-i-proper-pairing}, we have $\xi = \tau \circ r_{n-1}^k$ where $\tau = \sigma \circ (\sd d^{i-1})$.
  By \eqref{eq:j-face-j}, $\tau \in J_{n-1}^k$.
  If $\tau \in J_{n-1}^{k-1}$ then, by \eqref{eq:j-r-deg}, $\xi$ is degenerate: a contradiction as $\xi$ is assumed to be of type II.
  So $\tau \in J_{n-1}^k \backslash J_{n-1}^{k-1}$, hence $\xi = \tau \circ r_{n-1}^k$ is of type I, a contradiction.
  Secondly, let us suppose that $i \leq k$.
  Then, by \eqref{eq:r-lower-faces} and as in the proof of Lemma \ref{lem-only-type-i-proper-pairing}, we see that $\xi \in J_n^{k-1}$ and hence $\phi(\xi) < \phi(\sigma)$.
\end{proof}

\begin{corollary}
  The map $j : X \to \exinf X$ is a strong anodyne extension.
\end{corollary}

The following useful properties of $\exinf$ are well-known but not for \emph{strong} anodyne extensions.
(See \ref{def-classes} for the definitions of fibration and trivial fibration).

\begin{proposition}\label{prop-properties}
  $\exinf$ preserves: finite limits, monomorphisms, strong anodyne extensions, fibrations, trivial fibrations, simplicial homotopies, homotopy equivalences.
\end{proposition}
\begin{proof}
  Preservation of finite limits follows from the fact that $\exinf$ is a filtered colimit of right adjoints --- preservation of monomorphisms, homotopies and homotopy equivalences now follows easily.
  To see that $\exinf$ preserves trivial fibrations it is enough to see that $\sd$ sends the boundary inclusions to monomorphisms, but this is obvious.
  Preservation of fibrations is similar, using \ref{prop-subd-sae}.

  Finally, let $m : X \to Y$ be a strong anodyne extension and consider the diagram:
  \begin{displaymath}
    \begin{tikzpicture}
      \node (X) at (0,0) {$X$};
      \node (Y) at (0,-2) {$Y$};
      \node (ex) at (2.5,0) {$\exinf X$};
      \node (ey) at (4,-3) {$\exinf Y$};
      \node (un) at (2.5,-2) {$Y \cup \exinf X$};

      \draw[->,font=\footnotesize] (X)  edge (ex)
      edge node[auto,swap] {$m$} (Y)
      (ex) edge (un)
      edge[bend left=20] node[auto] {$\exinf m$}(ey)
      (Y)  edge (un)
      edge[bend right=20] (ey);
      \draw[->,dotted] (un) edge (ey);
      \draw (un) ++(-.5,.2) -- ++(0,0.3) -- ++(0.3,0);
    \end{tikzpicture}
  \end{displaymath}
  It is clear how every arrow in the diagram admits a regular pairing (transferring across pushouts in the obvious way), except possibly for $\exinf m$ and the dotted arrow.
  Moreover, the pairing on $Y \to Y \cup \exinf X$ is given by restriction from the one on $Y \to \exinf Y$.
  It is easy to check that in this situation the pairing on $Y \to \exinf Y$ also restricts to a pairing on the dotted arrow and that this restricted pairing is also regular.
  Hence $\exinf m$ is the composite of two strong anodyne extensions.
\end{proof}

The following proposition is the last difficult step in establishing the model structure axioms on $\sset$.
After this we rely only on elementary results found in the literature.

\begin{proposition}\label{prop-reflects}
  $\exinf$ reflects the triviality of fibrations.
  That is to say, if $f : X \to Y$ is a fibration and $\exinf f$ (a fibration by Proposition \ref{prop-properties}) is moreover a \emph{trivial} fibration, then $f$ is also a trivial fibration.
\end{proposition}
\begin{proof}
  Consider the lifting problem
  \begin{displaymath}
    \begin{tikzpicture}
      \node (B) at (0,0) {$\partial \triangle^n$};
      \node (T) at (0,-2) {$\triangle^n$};
      \node (X) at (2,0) {$X$};
      \node (Y) at (2,-2) {$Y$\rlap{ .}};
      \draw[->] (B) edge (T) edge node[auto] {$b$} (X) (X) edge (Y) (T) edge node[auto,swap] {$\sigma$} (Y);
      \draw[->,dotted] (T) -- (X);
    \end{tikzpicture}
  \end{displaymath}
  There is a diagonal filler if we compose this square with the naturality square for $j : \id \to \exinf$:
  \begin{displaymath}
    \begin{tikzpicture}
      \node (B) at (0,0) {$\partial \triangle^n$};
      \node (T) at (0,-2) {$\triangle^n$};
      \node (X) at (2,0) {$X$};
      \node (Y) at (2,-2) {$Y$};
      \node (eix) at (4,0) {$\exinf X$};
      \node (eiy) at (4,-2) {$\exinf Y$\rlap{ .}};
      \draw[->] (B) edge (T) edge node[auto] {$b$} (X) (X) edge (Y) edge (eix) (T) edge node[auto,swap] {$\sigma$} (Y) (Y) edge (eiy) (eix) edge (eiy);
      \draw[->,preaction={draw=white,line width=6pt}] (T) -- (eix);
    \end{tikzpicture}
  \end{displaymath}
  By finiteness of $\triangle^n$, we can replace $\exinf$ with $\ex^n$ for some $n$.
  By an inductive argument, we can assume that $n=1$.
  \begin{displaymath}
    \begin{tikzpicture}
      \node (B) at (0,0) {$\partial \triangle^n$};
      \node (T) at (0,-2) {$\triangle^n$};
      \node (X) at (2,0) {$X$};
      \node (Y) at (2,-2) {$Y$};
      \node (eix) at (4,0) {$\ex X$};
      \node (eiy) at (4,-2) {$\ex Y$\rlap{ .}};
      \draw[->] (B) edge (T) edge node[auto] {$b$} (X) (X) edge (Y) edge (eix) (T) edge node[auto,swap] {$\sigma$} (Y) (Y) edge (eiy) (eix) edge (eiy);
      \draw[->,preaction={draw=white,line width=6pt}] (T) -- node[auto,pos=0.3] {$\tau$} (eix);
    \end{tikzpicture}
  \end{displaymath}

  Let $X[\tau]$ be the `minimal anodyne subextension' of $X \to \ex X$ containing $\tau$.
  Letting $\tau = (\ex X)(S)(\upsilon)$, where $\upsilon$ is a nondegenerate $m$-simplex of $\ex X$ and $S : [n] \to [m]$ is a surjection in $\Delta$, this is the subcomplex whose nondegenerate simplices are
  \begin{displaymath}
    \nond{(X[\tau])} = \nond{(j_X(X))} \cup \{\xi \in \nond{(\ex X)}\backslash\nond{(j_X(X))} \mid \xi \ancpre \upsilon \}
  \end{displaymath}
  where $\ancpre$ is the ancestral preorder defined in \ref{def-ancestral-pre}.
  Clearly $X \to X[\tau]$ is a strong anodyne extension, since it admits a pairing by restriction from $X \to \ex X$ which is clearly regular.
  The crucial point now is that $X[\tau] \to \ex X \to \ex Y$ factorizes through $Y \to \ex Y$.
  We show that the set
  \begin{displaymath}
    D = \{ x \in \nond{(\ex X)} \mid f(x) \in J_{\dim x}^0(Y) \}
  \end{displaymath}
  is $\ancpre$-downwards closed and contains $\nond{(X)}$ and $\upsilon$.
  For downwards closure we need to show, for any $\rho \in J_p^0(Y)$, that each face $\rho \circ (\sd d^i)$ is in $J_{p-1}^0(Y)$ and $\rho \circ r_p^h$ is in $J_{p+1}^0(Y)$ for any $1 \leq h \leq p$.
  These are just special cases of equations \eqref{eq:j-face-j} and \eqref{eq:r-j-j-r} from Lemma \ref{lem-jr-eqns}.
  That $\nond{(X)} \subseteq D$ is obvious, and $\upsilon \in D$ since the image of $\upsilon$ is a face of the image of $\tau$ (which is the image in $\ex Y$ of $\sigma \in Y$), so $f(\upsilon) \in J_{\dim \upsilon}^0(Y)$ by \eqref{eq:j-face-j} of \ref{lem-jr-eqns}.

  Now consider the lifting problem
  \begin{displaymath}
    \begin{tikzpicture}
      \node (x1) at (0,0) {$X$};
      \node (xt) at (0,-2) {$X[\tau]$};
      \node (x2) at (2,0) {$X$};
      \node (y) at (2,-2) {$Y$\rlap{ .}};
      \draw[->] (x1) edge (xt) (xt) edge (y) (x2) edge node[auto] {$f$} (y);
      \draw[->] (x1) edge (x2);
      \draw[->,dotted] (xt) -- (x2);
    \end{tikzpicture}
  \end{displaymath}
  Since $f$ is a fibration and $X \to X[\tau]$ is a strong anodyne extension, this lifting problem has a solution.
  But the image of $\tau$ under this map is a solution to the original lifting problem.
\end{proof}

\section{The Model Structure}\label{sec-model-structure}

The definition of (closed) model category may be found in section I.9 of \cite{GOERSSJARDINE1999}.
We begin with the definition of the classes of morphisms.
It is easiest to describe all five classes and then subsequently to show that they are interrelated correctly, i.e.\ that $\tc = \ww \cap \cc = {}^\pitchfork \ff$, etc.

\begin{definition}\label{def-classes}
  The class $\ww$ of \emph{weak equivalences} is the class of all morphisms $f$ such that $\exinf f$ is a simplicial homotopy equivalence.
  The class $\cc$ of \emph{cofibrations} is the class of monomorphisms.
  The class $\tc$ of \emph{trivial cofibrations} is the class of morphisms which are retracts of strong anodyne extensions.
  The class $\ff$ of \emph{fibrations} is the class of morphisms with the right lifting property with respect to all horn inclusions.
  The class $\tf$ of \emph{trivial fibrations} is the class of morphisms with the right lifting property with respect to all monomorphisms, (or equivalently, to all boundary inclusions).
\end{definition}

\begin{lemma}\label{lem-weq}
  The weak equivalences have the 2-out-of-3 property and are closed under retracts.
\end{lemma}
\begin{proof}
  It is an easy exercise that simplicial homotopy equivalences have these properties.
\end{proof}

\begin{lemma}\label{lem-factorization}
  Any map $f : X \to Y$ can be factorized as a cofibration followed by a fibration, and we may choose either map to be trivial.
\end{lemma}
\begin{proof}
  This is Quillen's small object argument introduced in section II.\S 3 of \cite{QUILLEN1967}.
  Note that in the factorization into a trivial cofibration followed by a fibration we in fact get a strong anodyne extension for the first factor.
\end{proof}

\begin{lemma}\label{lem-saturation}
  The saturation axiom holds for these classes.
  That is, the class of trivial cofibrations is precisely the class of maps with the left lifting property with respect to the fibrations, as is the class of cofibrations to the trivial fibrations; and the class of trivial fibrations is precisely the class of maps with the right lifting property with respect to the cofibrations, as is the class of fibrations to the trivial cofibrations.
\end{lemma}
\begin{proof}
  The third assertion is nothing more than Definition \ref{def-classes}, and the fourth follows from the standard colimit closure properties of left classes, (see section IV.3 of \cite{GABRIELZISMAN1967}).
  The former two assertions are now easy consequences of \ref{lem-factorization} and the fact that monomorphisms are stable under retracts.
\end{proof}

\begin{remark}
  Lemma \ref{lem-saturation} is the only place in this argument where we essentially require the axiom of choice: for a given fibration, we choose lifts against every horn inclusion to show it has the right lifting property with respect to all strong anodyne extensions.
  We can avoid it by changing Definition \ref{def-classes} so that the fibrations are defined to admit such a simultaneous choice of liftings, or equivalently that they have the right lifting property with respect to all anodyne extensions (or just the strong ones).
  This is in practice no great inconvenience since this data is what we actually get from many natural constructions, including the small object argument.
  Now results such as \ref{prop-properties} can be seen as transporting the \emph{structure} of a Kan fibration on $f$ to the structure of a Kan fibration on $\exinf f$.
\end{remark}

To complete the proof, we need a few basic results which have short, elementary proofs in the literature.

\begin{lemma}\label{lemb-anod}
  An anodyne extension between Kan complexes is part of a simplicial homotopy equivalence.
\end{lemma}
\begin{proof}
  Proposition 3.2.3 in \cite{JOYALTIERNEY2008}
\end{proof}

\begin{lemma}\label{lemb-triv}
  Every trivial fibration is part of a simplicial homotopy equivalence.
\end{lemma}
\begin{proof}
  Proposition 3.2.5 in \cite{JOYALTIERNEY2008}.
\end{proof}

\begin{lemma}\label{lemb-fib}
  A fibration between Kan complexes which is part of a simplicial homotopy equivalence is a trivial fibration.
\end{lemma}
\begin{proof}
  Proposition 3.2.6 in \cite{JOYALTIERNEY2008}.
\end{proof}

\begin{theorem}
  The classes defined in \ref{def-classes} give rise to a model structure on $\sset$.
\end{theorem}
\begin{proof}
  It remains only to show that $\ww \cap \cc$ is indeed the class of trivial cofibrations as described in \ref{def-classes} and that $\ww \cap \ff$ is the class of trivial fibrations.
  Clearly every trivial fibration is a fibration and every trivial cofibration is a cofibration.
  We have four inclusions to prove: $\tf,\tc \subseteq \ww$, $\ww \cap \cc \subseteq \tc$, and $\ww \cap \ff \subseteq \tf$.
  
  Let $m : A \to B$ be a strong anodyne extension.
  Then by \ref{prop-properties}, $\exinf m$ is a strong anodyne extension.
  Now by \ref{lemb-anod}, $\exinf m$ is a simplicial homotopy equivalence, so $m$ is a weak equivalence.
  Any retract of $m$ is a weak equivalence by \ref{lem-weq}.

  Let $f : X \to Y$ be a trivial fibration.
  By \ref{lemb-triv} $f$ is a simplicial homotopy equivalence, so by \ref{prop-properties} $\exinf f$ is a simplicial homotopy equivalence and so $f$ is a weak equivalence.

  Now suppose $f : X \to Y$ is a fibration which is a weak equivalence.
  Then $\exinf f$ is a simplicial homotopy equivalence and by \ref{prop-properties} it is a fibration.
  Then by \ref{lemb-fib}, $\exinf f$ is a trivial fibration and by \ref{prop-reflects} $f$ is a trivial fibration.

  Finally suppose that $m : A \to B$ is a cofibration which is a weak equivalence.
  Then $m$ can be factorized as $m = p i$ where $i$ is a strong anodyne extension and $p$ is a fibration.
  By \ref{lem-weq}, $p$ is a weak equivalence and thus a trivial fibration by \ref{lemb-fib}.
  Hence one checks easily that $m$ is a retract of $i$.
\end{proof}

\section{Properness}\label{sec-properness}

A model category is said to be \emph{right-proper} when the class of weak equivalences is stable under pullback along fibrations.
The category $\sset$ equipped with the Kan-Quillen model structure is well known to satisfy both this condition and its dual, left-properness (weak equivalences stable under pushout along cofibrations).
See, for example, section II.8 of \cite{GOERSSJARDINE1999}, for a proof that left-properness follows formally from the fact that every simplicial set is cofibrant and that right-properness follows formally from the fact that $\exinf$ preserves fibrations and pullbacks.
In the remainder of this paper we give a new, elementary and direct proof of right-properness, which in fact leads us to a slightly stronger result.

Since trivial fibrations are stable under pullback along any map, it suffices to show that strong anodyne extensions are stable under pullback along fibrations.
Pullbacks in $\sset$ commute with colimits, so it will suffice to check:
\begin{theorem}\label{thm-proper}
  The pullback of a horn inclusion along a fibration is a \emph{strong} anodyne extension.
\end{theorem}
The traditional statement of right-properness amounts to Theorem \ref{thm-proper} with the word `strong' omitted.

\begin{notation}
  For the rest of the paper, we fix the following pullback square
  \begin{displaymath}
    \begin{tikzpicture}
      \node (A) at (0,2) {$A$};
      \node (B) at (0,0) {$B$};
      \node (h) at (2,2) {$\horn_k^n$};
      \node (t) at (2,0) {$\triangle^n$};
      \draw[->,font=\footnotesize] (A) -- node[auto,swap] {$m$} (B);
      \draw[->,font=\footnotesize] (B) -- node[auto,swap] {$f$} (t);
      \draw[->] (A) -- (h);
      \draw[->] (h) -- (t);
    \end{tikzpicture}
  \end{displaymath}
  in which $f$ is a fibration.
  We shall exhibit a regular pairing for $m$.
\end{notation}

The idea of the proof of Theorem \ref{thm-proper} is as follows.
Call $x \in B_{r+s}$ an $(r,s)$-simplex if there are precisely $s$ values of $i$ for which $d_ix = (k) \in \triangle^n$ (hence if $x \notin A_{r+s}$ we have $r \geq n$).
The pairing will pair each (nondegenerate) $(r,s)$-simplex (not in $A_{r+s}$) with either an $(r,s-1)$- or $(r,s+1)$-simplex, and we work recursively in $r$, adding all $(r,s)$-simplices before any $(r+1,s)$-simplices.
The first step in constructing the pairing is to assign each $(r,0)$-simplex $x$ an $(r,1)$-simplex $y$ such that $x$ is a face of $y$.
This is simple enough using the fact that $f$ is a fibration, and note that the remaining faces of $y$ are all $(r-1,1)$, so either in $A_r$ or already added.
However, this process leaves plenty of $(r,1)$-simplices $y$ unmatched, so we assign the remaining ones $(r,2)$-simplices $z$.
Now note that such a $z$ will have, in addition to $y$, another face $y'$ which is an $(r,1)$-simplex.
Hence we need to take care here, in case $y = y'$, or our process happens to match $y'$ with $z$ as well, or the eventual pairing fails to be regular --- we cannot use the fibrancy of $f$ na\"ively.
Thus from $y$ we choose $y'$ in advance to be either something we have already matched with an $(r,0)$-simplex in the first step, or something in the subcomplex $A$, or a degenerate simplex.
We use the fibrancy of $f$ to fill in the two simplices $y$ and $y'$, which share one maximal face, to an $(r,2)$-simplex.

We continue by recursion in $s$, but for larger $s$ the situation is more complicated.
When matching an $(r,s)$-simplex $x$ to an $(r,s+1)$-simplex $y$, $y$ will have $s$ faces that are $(r,s)$-simplices in addition to $x$.
We choose all of them in advance based on $x$ with the same requirement as before, but now we must rely on having had the foresight to have matched $(r,s-1)$-simplices with $(r,s)$-simplices in such a way that we can now take $s$ suitable $(r,s)$-simplices together with $x$ to form a compatible family suitable for Kan filling.
We now present the details in the framework of regular pairings, after introducing some notational conventions that will allow us to greatly simplify the simplicial identities.

\begin{notation}
  Let $\aaa$ be some fixed countable dense totally ordered set without greatest or least elements.
  We shall refer to the elements of $\aaa$ as \emph{names}.
  Given a non-empty finite subset $I \subseteq \aaa$ considered as a poset, let $\triangle^I$ denote the nerve of $I$.
  Note that $\triangle^I \cong \triangle^{\lvert I \rvert-1}$ in a unique way.

  Given such an $I$ and a simplicial set $X$, an \emph{$I$-simplex} of $X$ is a map $x : \triangle^I \to X$.
  The set of $I$-simplices is denoted $X_I$ --- clearly $X_I \cong X_{\lvert I \rvert-1}$.
  Given such an $I$ and $X$, if $a \in I$ denote by $d_a$ the function $X_I \to X_{I\backslash \{a\}}$ given by precomposition with the nerve of the poset inclusion $I\backslash \{a\} \to I$.
  It will be convenient to say, for any such $I$ and $a,b \in I$, that \emph{$b$ covers $a$ in $I$}, written $a \covered{I} b$, if $a < b$ and for all $x \in I$:
  \begin{displaymath}
    a \leq x \leq b \implies a = x \text{ or } x = b.
  \end{displaymath}
  Hence, when $b \in I$, $c \in \aaa\backslash I$ and $b \covered{I\cup\{c\}} c$, we denote by $s_b^c$ that function $X_I \to X_{I \cup \{c\}}$ given by precomposing with the nerve of the poset surjection $I \cup \{c\} \to I$ sending $c$ to $b$ and acting as the identity otherwise.
  Finally, for $a \in I$ let the \emph{$a$\textsuperscript{th} vertex map} $v_a : X_I \to X_{\{a\}} = X_0$ be the function given by precomposing with the nerve of the inclusion $\{a\} \hookrightarrow I$.
\end{notation}

The purpose of this notation is essentially to validate the observation that, for any $a,b \in I$ (with $a \neq b$) and $x \in X_I$, we have $d_a d_b x = d_b d_a x$.
The notation also simplifies the laws for swapping the order of degeneracy maps and for interchanging face and degeneracy maps.
As we go through the proof of \ref{thm-proper} we shall assume that simplices come with a suitable choice of names.

We explained above that the pairing $\parent$ will match $(r,s)$-simplices $x$ with $(r,s+1)$-simplices $y$ as long no $(r,s-1)$-simplex $w$ was matched with $x$ earlier.
It is simpler to define an operation $\qarent$ that assigns all $(r,s)$-simplices an $(r,s+1)$-simplex at first, and then to cut down the domain subsequently.
To be precise, for every (possibly degenerate) $n$-simplex $x$ in $B$ but not in $A$, we will assign an $(n+1)$-simplex $\qarent x$ which has $x$ as a face.
If we assume that $x$ is given as an $I$-simplex, then for some fresh name $z$ we will give $\qarent x$ as an $(I \cup \{z\})$-simplex, where $d_z \qarent x = x$ and $f(v_z x) = (k) \in (\triangle^n)_0$, and moreover $z$ is greater than all of the names $a \in I$ with $f(v_a x) = (k) \in (\triangle^n)_0$.
We will indicate the choice of name $z$ by writing $\qarent_z x$.
We require $\qarent$ to be a homomorphism with respect to face maps corresponding to vertices over $(k) \in (\triangle^n)_0$:
\begin{equation}
\label{k-faces-condition}\tag{$\dagger$} \text{if $a \in I$ with $f(v_a x) = (k) \in (\triangle^n)_0$, then $d_a(\qarent_z x) = \qarent_z(d_a x)$.}
\end{equation}
This is the compatibility condition that ensures Kan filling is possible using the following lemma.

\begin{lemma}\label{lem-union-of-faces}
  Let $n \geq 1$.
  For any $k$ with $1 \leq k \leq n$, the inclusion of a union of precisely $k$ codimension 1 faces of $\triangle^n$ is a strong anodyne extension.
\end{lemma}
\begin{proof}
  We use induction on $n$ --- the case $n = 1$ is trivial.
  Let $n > 1$, pick a set $I \subset \aaa$ of names with $\lvert I \rvert = n+1$, and let $J \subset I$ satisfy $1 \leq \lvert J \rvert \leq n = \lvert I \rvert - 1$.
  We consider the inclusion
  \begin{displaymath}
    \bigcup_{j \in J} \triangle^{I \backslash \{j\}} \hookrightarrow \triangle^I
  \end{displaymath}
  and proceed by a backwards induction on $\lvert J \rvert$.
  The case $\lvert J \rvert = n$ is trivial.
  Suppose $\lvert J \rvert < n$ and let $a \in I \backslash J$. Then, by the induction hypothesis for $n$, the inclusion
  \begin{displaymath}
    \bigcup_{j \in J} \triangle^{I \backslash \{a,j\}} \hookrightarrow \triangle^{I \backslash \{a\}}
  \end{displaymath}
  is strong anodyne.
  Hence its pushout, the inclusion
  \begin{displaymath}
    \bigcup_{j \in J} \triangle^{I \backslash \{j\}} \hookrightarrow \bigcup_{x \in J \cup \{a\}} \triangle^{I \backslash \{x\}},
  \end{displaymath}
  is also strong anodyne, so we may finish by applying the induction hypothesis for $|J|$ to
  \begin{displaymath}
    \bigcup_{x \in J \cup \{a\}} \triangle^{I \backslash \{x\}} \hookrightarrow \triangle^I
  \end{displaymath}
  and composing these two inclusions.
\end{proof}

\begin{definition}\label{def-qarent}
  The following clauses define $\qarent_z x$ by recursion over $\dim x$, where $x$ is a (possibly degenerate) $I$-simplex of $B$ not in $A$:
  \begin{itemize}
  \item [a)] if $x = \qarent_w y$, then
    \begin{displaymath}
      \qarent_z (\qarent_w y) = s_w^z(\qarent_w y),
    \end{displaymath}
  \item [b)] if $x = s_c^d y$, then
    \begin{displaymath}
      \qarent_z (s_c^d y) = s_c^d (\qarent_z y),
    \end{displaymath}
  \item [c)] and otherwise, if $x$ is nondegenerate and not already in the image of $\qarent$, let $J \subseteq I$ be the set of names in $x$ whose vertices lie over $k$.
    Then, assuming condition \eqref{k-faces-condition}, the set of simplices $\{x\} \cup \{\qarent_z(d_ax) \mid a \in J\}$ gives rise to a map $\triangle^I \cup \bigcup_{a \in J} \triangle^{(I\backslash \{a\}) \cup \{z\}} \to B$, where the domain is a subcomplex of $\triangle^{I\cup\{z\}}$ given as a union of $\lvert J \rvert + 1$ codimension 1 faces.
    Since $x$ is not in $A$ we see that $J \neq I$ and hence that $0 < \lvert J\rvert + 1 < \lvert I \rvert + 1$.
    Now we may use Lemma \ref{lem-union-of-faces} and the fact that $f$ is a fibration to deduce that this map extends to one defined on the whole of $\triangle^{I \cup \{z\}}$.
    Let $\qarent_z(x)$ be the $(I \cup \{z\})$-simplex corresponding to some choice of such an extension.
  \end{itemize}
\end{definition}

Note that the hypotheses of clauses a) and b) are not mutually exclusive, but in every clause the definition of $\qarent_z x$ depends on the values of $\qarent y$ only for simplices $y$ with $\dim y < \dim x$.
There are some things to check before we can state that \ref{def-qarent} gives us a well-defined function (up to choice of fillers for the fibration $f$).
Firstly, for clause c) we need the following lemma.

\begin{lemma}\label{lem-qarent-faces}
  Let $n \geq 1$ and suppose we are given $\qarent$ defined on all simplices $x \in B$ with $x \notin A$ and $\dim x \leq n$, satisfying the three clauses of \ref{def-qarent} for every such $x$, and satisfying \eqref{k-faces-condition} for those $x$ with $\dim x \leq n-1$.
  Then \eqref{k-faces-condition} also holds for every $x \in B \backslash A$ with $\dim x = n$.
\end{lemma}
\begin{proof}
  Easy inspection.
\end{proof}

To complete the verification that \ref{def-qarent} does define a function, we must check that clause b) gives the same output for any $x$ which is degenerate in two different ways, and also that any simplex matching the hypotheses of clauses a) and b) is given the same image under $\qarent$ by either clause.
We check these points in the following two lemmas.

\begin{lemma}
  Let $X$ be a simplicial set, $I \subseteq \aaa$ a finite set of names with $a,b,c,d \in I$ where $a \covered{I} b$ and $c \covered{I} d$, $x$ an $I \backslash \{b\}$-simplex and $y$ an $I \backslash \{d\}$-simplex.
  Suppose $s_a^b x = s_c^d y$.
  Then either $a = c$ (hence $b = d$ and $x = y$) or there exists a unique $I \backslash \{b,d\}$-simplex $w$ with $x = s_c^d w$ and $y = s_a^b w$.
\end{lemma}
\begin{proof}
  Easy calculation: without loss of generality $b \leq c$ rather than $d \leq a$.
  Then take $w = d_d x = d_a y$ when $b = c$ and $w = d_c x = d_a y$ when $b < c$.
\end{proof}

\begin{lemma}
  Let $n \geq 1$ and suppose we are given $\qarent$ defined on all simplices $x \in B$ with $x \notin A$ and $\dim x \leq n$ and satisfying the three clauses of \ref{def-qarent} and \eqref{k-faces-condition} for every such $x$.
  Now suppose $I \subseteq \aaa$ is a set of names with $\lvert I \rvert = n + 2$, we have names $a,b,z \in I$ with $a \covered{I} b$, and $x \in B_{I \backslash \{z\}} \backslash A_{I \backslash \{z\}}$ and $y \in B_{I \backslash \{b\}} \backslash A_{I \backslash \{b\}}$ are two $n$-simplices in $B$ but not in $A$.
  Suppose that these simplices with these names satisfy $\qarent_z x = s_a^b y$.
  Then for any appropriate fresh name $w$, we have $s_z^w(\qarent_z x) = s_a^b(\qarent_w y)$, i.e.\ clauses a) and b) of \ref{def-qarent} agree on the value of $\qarent_w$ on the $(n+1)$-simplex $\qarent_z x = s_a^b y$.
\end{lemma}
\begin{proof}
  We cannot have $z = a$ since, in $\qarent_z x$, $z$ is the greatest name with $f(v_z(x)) = (k) \in (\triangle^n)_0$.
  If $z \neq b$, then $x = d_z s_a^b y = s_a^b d_z y$, so $s_a^b y = s_a^b \qarent_z(d_z y)$, so $y = \qarent_z (d_z y)$.
  Hence $\qarent_w y = s_z^w y$, from which the result follows.
  Alternatively, if $z = b$, then $y = d_z(\qarent_z x) = x$, and the result follows easily.
\end{proof}

We have shown that we can indeed give a function $\qarent$ satisfying the clauses of \ref{def-qarent} and condition \eqref{k-faces-condition}, so we are now in a position to give the pairing on $m : A \hookrightarrow B$.
\begin{definition}
  A nondegenerate simplex $x \in \nond{B}\backslash \nond{A}$ is of type I if and only if it is in the image of $\qarent$.
  For a type II simplex $y$, we define $\parent y = \qarent y$.
\end{definition}

\begin{proof}[of \ref{thm-proper}]
  Let us check that this is indeed a pairing on $m$.
  If $x$ is of type II then $d_z (\qarent_z x) = x$.
  Suppose $d_a (\qarent_z x) = x$ for some name $a \neq z$.
  Then $f(v_ax) = (k) \in (\triangle^n)_0$.
  So by \ref{lem-qarent-faces}, $x = d_a(\qarent_z x) = \qarent_z(d_a x)$, a contradiction, since $x$ was assumed not to be in the image of $\qarent$.
  Hence $\qarent_z x$ is nondegenerate and hence is of type I.
  If instead we suppose that $y = \qarent x$ is of type I, then $x$ is nondegenerate and not in the image of $\qarent$ since otherwise \ref{def-qarent} would force $y$ to be degenerate.
  The above also shows that we do indeed have a proper pairing.

  It remains to check that the pairing is regular.
  Given an $I$-simplex $x \in B_I$ of type II, define
  \begin{displaymath}
    \phi(x) = \big\lvert \{a \in I \mid f(v_a(x)) \neq (k) \in (\triangle^n)_0\} \big\rvert.
  \end{displaymath}
  Now suppose that $x \modanc n y$.
  We will show that $\phi(x) < \phi(y)$.
  Still using $I$ as the set of names in $x$, since $x \neq y$ we may pick names $z \in I$ and $a \in \aaa \backslash I$ such that $y$ can be modelled as an $(I\backslash \{z\}) \cup \{a\}$-simplex with $x = d_a(\parent_z y)$.
  If $f(v_a(y)) = (k) \in (\triangle^n)_0$, then $x = d_a(\parent_z y) = \qarent_z(d_a x)$ is in the image of $\qarent$ and so is of type I, a contradiction.
  Hence $f(v_a(y)) \neq (k) \in (\triangle^n)_0$, so $\phi(x) = \phi(y) - 1$.
  Thus, by \ref{lem-modanc} and \ref{prop-structure-to-presentation}, $m$ is a strong anodyne extension.
\end{proof}

\begin{remark}
  Let us conclude by remarking that we can avoid the axiom of choice by considering fibrancy as a structure rather than a property.
  Suppose we have a fibration $f : Y \to X$ and a fibrant object in $\sset/Y$, i.e.\ a fibration $g : Z \to Y$.
  Suppose moreover that $f$ and $g$ are equipped with choices of solution to every horn-filling problem.
  Then $\Pi_f(g) \in \sset/X$ may be equipped with a choice of solution to every horn-filling problem.
  One uses the adjunction $f^* \dashv \Pi_f : \sset/Y \to \sset/X$ and the fact that \ref{thm-proper} tells us that horn inclusions in $\sset/X$ are sent by $f^*$ to strong anodyne extensions with an explicit anodyne presentation derived from the fibration structure on $f$.
\end{remark}

\section*{Acknowledgements}
I would like to thank Martin Hyland and Zhen Lin Low for their encouragement and many illuminating conversations.
I thank the anonymous referee for helpful feedback and pointing out the connection to Forman's discrete Morse Theory.

\end
{document}